\documentclass[10pt]{amsart}
\usepackage{graphicx}
\usepackage{amsmath,amssymb,xy,array}
\usepackage[T1]{fontenc}
\usepackage{amsfonts}
\usepackage{mathabx}
\usepackage{amsmath}
\usepackage{amssymb}
\usepackage{amsthm}
\usepackage{graphicx}
\usepackage{enumitem}
\usepackage{xcolor}
\usepackage[utf8]{inputenc}
\usepackage{hyperref}
\usepackage{mathrsfs}
\usepackage{comment}
\usepackage[toc,page]{appendix}
\usepackage{fancyhdr}
\usepackage{tikz}
\usepackage{tikz-cd}
\usepackage{longtable}
\usepackage{geometry}
\usepackage{textcomp}
\usetikzlibrary{trees}
\usetikzlibrary{arrows}
\usepackage{multirow}
\usepackage{youngtab}
\usepackage{mathdots}
\usetikzlibrary{arrows.meta}

\definecolor{yqyqyq}{rgb}{0.5019607843137255,0.5019607843137255,0.5019607843137255}\definecolor{uuuuuu}{rgb}{0.26666666666666666,0.26666666666666666,0.26666666666666666}
\definecolor{uququq}{rgb}{0.25098039215686274,0.25098039215686274,0.25098039215686274}
\definecolor{wwwwww}{rgb}{0.4,0.4,0.4}
\definecolor{uuuuuu}{rgb}{0.26666666666666666,0.26666666666666666,0.26666666666666666}

\setlist[itemize]{leftmargin=6mm}

\renewcommand{\P}{\mathbb P}

\DeclareMathOperator{\codim}{codim}

\newcommand{\Aut}{\operatorname{Aut}}
\newcommand{\PsAut}{\operatorname{PsAut}}

\DeclareMathOperator{\Cl}{Cl}

\DeclareMathOperator{\lin}{lin}

\DeclareMathOperator{\Hilb}{Hilb}

\DeclareMathOperator{\mult}{mult}

\DeclareMathOperator{\Hom}{Hom}

\DeclareMathOperator{\Exc}{Exc}

\DeclareMathOperator{\Sing}{Sing}

\DeclareMathOperator{\Eff}{Eff}
\DeclareMathOperator{\Nef}{Nef}
\DeclareMathOperator{\Mov}{Mov}
\DeclareMathOperator{\Pic}{Pic}
\DeclareMathOperator{\rank}{rank}

\renewcommand{\sec}{\mathbb{S}ec}

\DeclareMathOperator{\Sym}{Sym}
\DeclareMathOperator{\Cox}{Cox}

\DeclareMathOperator{\diag}{diag}

\renewcommand{\P}{\mathbb{P}}

\newcommand{\xn}{X_{2r}}
\newcommand{\xnb}{\mathcal{S}_{2r}}

\newcommand{\spn}{Sp(2r)}
\newcommand{\vr}{\mathcal{V}^{2r-1}_2}
\newcommand{\mm}{M_{2r,2r}(K)}
\newcommand{\glnm}{GL(r)}
\newcommand{\sod}{SO(2)}
\newcommand{\son}{SO(2r)}

\newcommand{\gln}{GL(2r)}
\newcommand{\mmle}{\overline{M}_{0,0}(LG(2,4),2)}
\newcommand{\mml}{\overline{M}_{0,0}(LG(r,2r),2)}

\newtheorem{thm}{Theorem}[section]

\newtheorem{Lemma}[thm]{Lemma}
\newtheorem{Proposition}[thm]{Proposition}

\newtheorem{Corollary}[thm]{Corollary}

\theoremstyle{definition}

\newtheorem{Definition}[thm]{Definition}
\newtheorem{Remark}[thm]{Remark}

\newtheorem{Example}[thm]{Example}
\newtheorem{Notation}[thm]{Notation}

\newtheorem{Construction}[thm]{Construction}

\hypersetup{pdfpagemode=UseNone}
\hypersetup{pdfstartview=FitH}

\graphicspath{{immagini/}}
\usepackage{pgf,tikz,pgfplots}
\pgfplotsset{compat=1.15}
\usepackage{mathrsfs}
\usetikzlibrary{arrows}

\begin{document}

\title[\resizebox{6.1in}{!}{Complete symplectic quadrics and Kontsevich spaces of conics in Lagrangian Grassmannians}]{Complete symplectic quadrics and Kontsevich spaces of conics in Lagrangian Grassmannians}

\author[Elsa Corniani]{Elsa Corniani}
\address{\sc Elsa Corniani\\ Dipartimento di Matematica e Informatica, Universit\`a di Ferrara, Via Machiavelli 30, 44121 Ferrara, Italy}
\email{elsa.corniani@unife.it}

\author[Alex Massarenti]{Alex Massarenti}
\address{\sc Alex Massarenti\\ Dipartimento di Matematica e Informatica, Universit\`a di Ferrara, Via Machiavelli 30, 44121 Ferrara, Italy}
\email{alex.massarenti@unife.it}

\date{\today}
\subjclass[2010]{Primary 14M27, 14E30; Secondary 14J45, 14N05, 14E07}
\keywords{Wonderful compactifications; Mori dream spaces; Cox rings; Spherical varieties; Stable maps}

\maketitle

\begin{abstract}
A wonderful compactification of an orbit under the action of a semi-simple and simply connected group is a smooth projective variety containing the orbit as a dense open subset, and where the added boundary divisor is simple normal crossing. We construct the wonderful compactification of the space of symmetric and symplectic matrices, and investigate its geometry. As an application, we describe the birational geometry of the Kontsevich spaces parametrizing conics in Lagrangian Grassmannians. 
\end{abstract}

\setcounter{tocdepth}{1}


\section{Introduction}
The \textit{wonderful compactification} of a symmetric space was introduced by C. De Concini and C. Procesi in \cite{DP83}. Later on, D. Luna gave a more general definition of wonderful variety and then he proved that, according to his definition, all wonderful varieties are spherical \cite{Lu96}. 

Let $\mathscr{G}$ be a reductive group, and $\mathscr{B}\subset\mathscr{G}$ a Borel subgroup. A \textit{spherical variety} is a variety admitting an action of $\mathscr{G}$ with an open dense $\mathscr{B}$-orbit. For \textit{wonderful varieties} we require in addition the existence of an open orbit whose complementary set is a simple normal crossing divisor $D_1\cup\dots\cup D_r$, where the $D_i$ are the $\mathscr{G}$-invariant prime divisors in $X$. The number $r$ is called the rank of $X$. Note that $\mathscr{G}$ has $2^r$ orbits in $X$ given by all the possible intersections among the $D_i$. The unique closed orbit is $\bigcap_{i=1}^rD_i$. 

Apart from their role in group theory, wonderful varieties proved themselves important in enumerative geometry and recently also in birational geometry. We refer to \cite{BL11}, \cite{Pe14}, \cite{Pe18} for comprehensive treatments of these topics.

Classical examples of wonderful varieties are the spaces of complete quadrics and of complete collineations. These spaces have been studied both from the geometrical and enumerative point of view \cite{Se48}, \cite{Se51}, \cite{Se52}, \cite{Ty56}, \cite{Va82}, \cite{Va84}, \cite{TK88}, \cite{LLT89}, \cite{Tha99}. An aspect that will be fundamental in this paper is that spaces of complete quadrics and collineations play a role in the study of other moduli spaces such as Hilbert schemes and Kontsevich spaces of stable maps \cite{Al56}, \cite{Pi81}, \cite{Ca16}. The birational geometry of the spaces of complete quadrics and collineations, mostly from the point of view of Mori theory, has recently been studied in \cite{Ce15}, \cite{Ma18a}, \cite{Ma18b}.

The spaces of complete collineations and quadrics have been constructed, as a sequence of blow-ups, by I. Vainsencher in \cite{Va84}, \cite{Va82}, and a similar construction for complete skew-forms has been carried out by M. Thaddeus in \cite{Tha99}. In this paper we construct the wonderful compactification of the space of symmetric and symplectic matrices. More precisely, we summarize our main results in Propositions \ref{orb_dim}, \ref{fun}, and Theorem \ref{main1} as follows:

\begin{thm}\label{A}
Let $\mathbb{P}^N$ be the projective space parametrizing $2r\times 2r$ symmetric matrices modulo scalar, consider the following $\spn$-action: 
$$
\begin{array}{ccc}
\spn\times \mathbb{P}^{N} & \longrightarrow & \mathbb{P}^{N}\\
(M,Z) & \longmapsto & MZM^{t}
\end{array}
$$
and denote by $X_{2r}\subset\mathbb{P}^N$ the closure of the $\spn$-orbit of the identity. Then $X_{2r}$ admits a stratification 
$$Y_1\subset Y_2\subset\dots Y_r\subset X_{2r}$$
where the variety $Y_k$ parametrizes matrices in $X_{2r}$ of rank at most $k$, $\dim(Y_k) = 2rk+k-k^2-1$ for $k = 1,\dots,r$, and $\dim(X_{2r}) = r(r+1)$.

Furthermore, consider the following sequence of blow-ups 
$$\mathcal{S}_{2r}:=X_{2r}^{(r-1)}\rightarrow X_{2r}^{(r-2)}\rightarrow X_{2r}^{(r-3)}\rightarrow\dots\rightarrow X_{2r}^{(1)}\rightarrow X_{2r}^{(0)}:=X_{2r}$$
where $X_{2r}^{(k)}\rightarrow X_{2r}^{(k-1)}$ is the blow-up of the strict transform of $Y_k$ in $X_{2r}^{(k-1)}$ for $k = 1,\dots, r-1$. Denote by $E_k\subset\mathcal{S}_{2r}$ the exceptional divisor over $Y_k$ for $k=1,\dots,r-1$, and by $S_r^{(r-1)}(\mathcal{V}_2^{2r-1})$ the strict transform of the divisor $Y_r\subset X_{2r}$. Then $E_1,\dots,E_{r-1},S_r^{(r-1)}(\mathcal{V}_2^{2r-1})$ are smooth and intersect transversally. Furthermore, the closures of the orbits of the $Sp(2r)$-action on $\mathcal{S}_{2r}$ induced by the $\spn$-action above are given by all the possible intersections among $E_1,\dots,E_{r-1},S_r^{(r-1)}(\mathcal{V}_2^{2r-1})$ and $\mathcal{S}_{2r}$ itself. Therefore $\mathcal{S}_{2r}$ is wonderful. 
\end{thm}
We will call $\mathcal{S}_{2r}$ the space of \textit{complete symplectic quadrics} of dimension $2r-2$. By Proposition \ref{fun} $Y_k$ is the intersection of $X_{2r}$ with the secant variety $\sec_k(\mathcal{V}_2^{2r-1})$ that is the closure of the union of the $(k-1)$-planes generated by $k$ general points on the Veronese variety $\mathcal{V}_2^{2r-1}$ of degree two and dimension $2r-1$. 

Note that the formula for the dimension of $Y_k$ in Theorem \ref{A} yields that $\mathcal{V}_2^{2r-1}$ is entirely contained in $X_{2r}$, while for $r\geq 2$ the orbit closure $X_{2r}$ intersects $\sec_k(\mathcal{V}_2^{2r-1})$ in a proper subvariety. Furthermore, by Proposition \ref{sec} we have that set-theoretically $\sec_k(\mathcal{V}_2^{2r-1})\cap X_{2r} = \sec_r(\mathcal{V}_2^{2r-1})\cap X_{2r}$ for $k\geq r$. Interestingly, this means that if $M$ is a symmetric $2r\times 2r$ matrix that is a limit of a family of symplectic matrices then either $1\leq \rank(M)\leq r$ or $\rank(M) = 2r$.

For instance, by Proposition \ref{x4g14} $X_4$ is the Grassmannian $\mathbb{G}(1,4)$ of lines in $\mathbb{P}^4$. In this case by Theorem \ref{A} we have that $\mathcal{S}_4$ is the blow-up of $\mathbb{G}(1,4)$ along the Veronese $3$-fold $\mathcal{V}_2^3\subset\mathbb{G}(1,4)$. This is a wonderful variety of rank two. As remarked in \cite{Wa96} wonderful varieties of rank two are a building block in the theory of spherical varieties. The wonderful compactification $\mathcal{S}_4$ is the sixth variety in \cite[Table C]{Wa96}, and will be a central character throughout the whole paper.

\begin{Remark}\label{Rem}
The use of wonderful compactifications in enumerative geometry dates back to the solution of M. Chasles to a problem posed by J. Steiner asking how many conics in the plane are tangent to five given general conics \cite{KL80}. Steiner's answer, which then turned out to be wrong, was $6^5 = 7776$. Later on Chasles computed the right number which is $3264$.

Although enumerative problems are not within the scope of this paper, we give a simple application of our construction in enumerative geometry. It is well known that there are $92$ quadric surfaces in $\mathbb{P}^3$ that are tangent to nine general lines \cite[Remark 4.3]{BFS20}. The points of $\mathcal{S}_4$ in a divisor of class $2H-E_1$, where $H$ is the pull-back of the hyperplane class of $X_4$, correspond to the symplectic quadrics in $\mathbb{P}^3$ that are tangent to a general line. We have that $(2H-E_1)^6 = 40$. From the enumerative point of view this means that there are exactly $40$ symplectic quadrics in $\mathbb{P}^3$ that are tangent to six general lines.    
\end{Remark}

The variety $X_{2r}$ is singular for $r\geq 3$. The wonderful variety $\mathcal{S}_{2r}$ may be seen as an incarnation, in the singular setting, of the process producing a wonderful compactification from a conical one in \cite{MP98}. Furthermore, by Proposition \ref{symplcones} $\mathcal{S}_{2r}$ provides a resolution of a variety with conical singularities as remarked in \cite[Section 3.3]{MP98}.  

In Section \ref{divS2r} and \ref{birS2r} we take advantage of the spherical structure of $\mathcal{S}_{2r}$ to study its birational geometry from the point of view of Mori theory. Roughly speaking, a \textit{Mori dream space} is a projective variety $X$ whose cone of effective divisors $\Eff(X)$ admits a well-behaved decomposition into convex sets, called Mori chamber decomposition, and these chambers are the nef cones of birational models of $X$. These varieties, introduced by Y. Hu and S. Keel in \cite{HK00}, are named so because they behave in the best possible way from the point of view of the minimal model program. In general, to determine whether or not a variety is a Mori dream space, and in case to study in detail its Mori chamber decomposition is a hard problem. This has been done for instance when $X$ is obtained by blowing-up points in a projective space \cite{Mu01}, \cite{CT06}, \cite{AM16}, \cite{AC17}, \cite{BM17}, \cite{LP17}.

Spherical varieties are Mori dream spaces, we refer to \cite{Pe14} for a comprehensive treatment of these topics. Cox rings were first introduced by D. A. Cox for toric varieties \cite{Cox95}, and then his construction was generalized to projective varieties in \cite{HK00}. These algebraic objects are basically universal homogeneous coordinate rings of projective varieties, defined as the direct sum of the spaces of sections of all isomorphism classes of line bundles on them. We have that a normal $\mathbb{Q}$-factorial projective variety $X$, over an algebraically closed field, with finitely generated Picard group is a Mori dream space if and only if its Cox ring is finitely generated \cite[Proposition 2.9]{HK00}. Summing-up the results in Propositions \ref{pic_x2r}, \ref{eff_nef}, \ref{mcd_4} and Theorem \ref{dec_S6} we have the following: 

\begin{thm}\label{B}
Fix homogeneous coordinates $[z_{0,0}:\dots:z_{n,n}]$ on $\mathbb{P}^N$, and consider the blow-up $f:\mathcal{S}_{2r}\rightarrow X_{2r}\subset\mathbb{P}^N$ with exceptional divisors $E_1,\dots,E_{r-1}$ in Theorem \ref{A}. For $i=1,\dots,r$ we define the divisors $D_i$ as the strict transforms in $\mathcal{S}_{2r}$ of the divisor given by the intersection of  
$$\det \begin{pmatrix}
z_{0,0} & \dots & z_{0,i-1}\\
\vdots & \ddots & \vdots \\
z_{0,i-1} & \dots & z_{i-1,i-1}\\
\end{pmatrix}=0$$
with $X_{2r}$, and let $H$ be the pull-back of the hyperplane section of $X_{2r}\subset\mathbb{P}^N$ to $\mathcal{S}_{2r}$. 

The Picard rank of $\mathcal{S}_{2r}$ is $\rho(\mathcal{S}_{2r}) = r$ and $\Pic(\mathcal{S}_{2r})$ is generated by $H,E_1,\dots,E_{r-1}$. Furthermore, the effective cone $\Eff(\mathcal{S}_{2r})$ is generated by $E_1,\dots,E_{r-1},S_r^{(r-1)}(\mathcal{V}_2^{2r-1})$, the nef cone $\Nef(\mathcal{S}_{2r})$ is generated by $D_1,\dots,D_r$, and the Cox ring of $\mathcal{S}_{2r}$ is generated by the sections of $E_1,\dots,E_{r-1},S_r^{(r-1)}(\mathcal{V}_2^{2r-1}), D_1,\dots,D_r$.

Finally, the Mori chamber decomposition of the $\Eff(\mathcal{S}_{4})$ has three chambers, and the Mori chamber decomposition of the $\Eff(\mathcal{S}_{6})$ has nine chambers. 
\end{thm}        
We refer to Proposition \ref{mcd_4} and Theorem \ref{dec_S6} for a detailed description of the Mori chamber decompositions. 

In Section \ref{K_Grass} we investigate the birational geometry of Kontsevich moduli spaces of conics in Lagrangian Grassmannians. These spaces are denoted by $\overline{M}_{g,n}(X,\beta)$ where $X$ is a projective scheme and $\beta\in H_2(X,\mathbb{Z})$ is the homology class of a curve in $X$. A point in $\overline{M}_{g,n}(X,\beta)$ corresponds to a holomorphic map $\alpha$ from an $n$-pointed genus $g$ curve $C$ to $X$ such that $\alpha_{*}([C])=\beta$. If $X$ is a homogeneous variety then there exists a smooth, irreducible Deligne-Mumford stack $\overline{\mathcal{M}}_{0,n}(X,\beta)$ whose coarse moduli space is $\overline{M}_{0,n}(X,\beta)$ \cite{FP}. When $X$ is a Lagrangian Grassmannian the class $\beta$ is then completely determined by its degree and we will write $\beta = d[L]$, where $[L]$ is the class of a line in the Pl\"ucker embedding. The Mori theory of the spaces $\overline{M}_{0,n}(X,\beta)$, especially when the target variety is a projective space or a Grassmannian, has been widely investigated in a series of papers \cite{CS06}, \cite{Ch08}, \cite{CHS08}, \cite{CHS09}, \cite{CC10}, \cite{CC11}, \cite{CM17}.

On the Kontsevich space $\overline{M}_{0,0}(LG(r,2r),2)$ of conics in the Lagrangian Grassmannian, $LG(r,2r)$ parametrizing Lagrangian subspaces of a $2r$-dimensional symplectic vector space, we consider the divisor classes: $\Delta^r$ of maps with reducible domain, $T^r$ of conics tangent to a fixed hyperplane section of $LG(r,2r)$, $H^r_{\sigma_2}$ of conics intersecting a fixed codimension two Schubert variety $\Sigma_2^r\subset LG(r,2r)$, and $D_{unb}^r$ which we now define. A stable map $\alpha:\mathbb{P}^1\rightarrow LG(r,2r)$ induces a rank two subbundle $\mathcal{E}_{\alpha}\subset \mathcal{O}_{\mathbb{P}^1}\otimes K^{2r}$. If $r = 2$ we define $D_{unb}$ as the closure of the locus of maps $[\mathbb{P}^1,\alpha]\in \overline{M}_{0,0}(LG(2,4),2)$ such that $\mathcal{E}_{\alpha}\neq \mathcal{O}_{\mathbb{P}^1}(-1)^{\oplus 2}$. If $r\geq 3$ there is a trivial subbundle $\mathcal{O}_{\mathbb{P}^1}^{\oplus r-2}\subset\mathcal{E}_{\alpha}$ which induces a $(r-2)$-dimensional subspace $H_{\alpha}\subset\mathbb{P}^{2r-1}$. We define $D_{unb}^r$ as the closure of the locus of maps $[\mathbb{P}^1,\alpha]\in\overline{M}_{0,0}(LG(r,2r),2)$ such that $H_{\alpha}$ intersects a fixed $(r+1)$-dimensional subspace of $\mathbb{P}^{2r-1}$.

The main results in Lemma \ref{sphKLG}, Proposition \ref{effneflg}, Theorem \ref{mcd_lg}, Remark \ref{contr4} and Corollary \ref{Fano} can be summarized in the following statement: 

\begin{thm}\label{C}
Let $\overline{M}_{0,0}(LG(r,2r),2)$ be the Kontsevich space of conics in the Lagrangian Grassmannian $LG(r,2r)$, parametrizing Lagrangian subspaces of a $2r$-dimensional symplectic vector space, with $r\geq 2$. 

The effective cone $\Eff(\overline{M}_{0,0}(LG(r,2r),2))$ is generated by $\Delta^r$ and $D_{unb}^r$, and the nef cone $\Nef(\overline{M}_{0,0}(LG(r,2r),2))$ is generated by $H_{\sigma_2}^r$ and $T^r$. 

The Mori chamber decomposition of $\Eff(\mml)$ has three chambers as displayed in the following picture:
$$
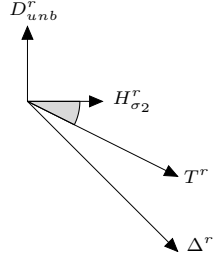
\begin{tikzpicture}[line cap=round,line join=round,>=triangle 45,x=1.0cm,y=1.0cm]
xmin=2.5,
xmax=7.5,
ymin=1.7032313370047312,
ymax=5.301934941656083,
xtick={2.5,3.0,...,5.5},
ytick={2.0,2.5,...,5.0},]
\clip(2.1,2.0) rectangle (7.5,5.4);
\draw [->,line width=0.1pt] (3.,4.) -- (3.,5.); 
\draw [->,line width=0.1pt] (3.,4.) -- (4.,4.); 
\draw [->,line width=0.1pt] (3.,4.) -- (5.,3.); 
\draw [->,line width=0.1pt] (3.,4.) -- (5.,2.); 
\draw [shift={(3.,4.)},line width=0.4pt,fill=black,fill opacity=0.15000000596046448]  (0,0) --  plot[domain=-0.46364760900080615:0.,variable=\t]({1.*0.6965067581669063*cos(\t r)+0.*0.6965067581669063*sin(\t r)},{0.*0.6965067581669063*cos(\t r)+1.*0.6965067581669063*sin(\t r)}) -- cycle ;
\begin{scriptsize}
\draw[color=black] (3.085808915158281,5.2) node {$D_{unb}^r$};
\draw[color=black] (4.4,4) node {$H_{\sigma_2}^r$};
\draw[color=black] (5.25,3) node {$T^r$};
\draw[color=black] (5.3,2.1) node {$\Delta^r$};
\end{scriptsize}
\end{tikzpicture}
$$
where $H_{\sigma_2}^{r}\sim \frac{1}{2}(\Delta^r+2D_{unb}^r)$ and $T^r\sim\Delta^r+D_{unb}^r$. Furthermore, if $r\geq 2$ then $\Mov(\overline{M}_{0,0}(LG(r,2r),2))$ is generated by $T^r$ and $D_{unb}^r$ while $\Mov(\overline{M}_{0,0}(LG(2,4),2))$ is generated by $T^r$ and $H_{\sigma_2}^r$.

The divisor $H_{\sigma_2}^r$ induces a birational morphism 
$$f_{H_{\sigma_2}^r}:\mml\rightarrow \widetilde{Chow}(LG(r,2r),2)$$ 
which is an isomorphism away form the locus $Q^r(1)$ of double covers of a line in $LG(r,2r)$, and contracts $Q^r(1)$ so that the locus of double covers with the same image maps to
a point, where $\widetilde{Chow}(LG(r,2r),2)$ is the normalization of the Chow variety of conics in $LG(r,2r)$.

The divisor $T^r$ induces a morphism 
$$f_{T^r}:\mml\rightarrow \overline{M}_{0,0}(LG(r,2r),2,1)$$ 
which is an isomorphism away from $\Delta^r$ and contracts the locus of maps with reducible domain $[C_1\cup C_2,\alpha]$ to $\alpha(C_1\cap C_2)$, where $\overline{M}_{0,0}(LG(r,2r),2,1)$ is the moduli space of weighted stable maps to $LG(r,2r)$.

The birational model $X_r$ corresponding to the chamber delimited by $H_{\sigma_2}^r$ and $D_{unb}^r$ is a fibration $X_r\rightarrow SG(r-2,2r)$ with fibers isomorphic to the Grassmannian $\mathbb{G}(2,4)$ parametrizing plane in $\mathbb{P}^4$, where $SG(r-2,2r)$ is the symplectic Grassmannian parametrizing isotropic subspaces of dimension $r-2$. Moreover, $D_{unb}^r$ contracts $\overline{M}_{0,0}(LG(r,2r),2)$ onto $SG(r-2,2r)$.

Finally, $\mml$ is Fano for $2\leq r\leq 6$, weak Fano, that is $-K_{\mml}$ is nef and big, for $r = 7$, and $-K_{\mml}$ is not ample for $r\geq 8$.
\end{thm}

Moreover, Proposition \ref{isor2}, Remarks \ref{Rem}, \ref{contr4} and Corollary \ref{aut} provide additional information for the case $r = 2$.
\begin{thm}\label{D}
The following $Sp(4)$-action  
$$
\begin{array}{cll}
Sp(4) \times \overline{M}_{0,0}(LG(2,4),2) & \longrightarrow & \overline{M}_{0,0}(LG(2,4),2) \\ 
(M, [C, \alpha]) & \longmapsto & [C, \wedge^{2}M\circ\alpha] 
\end{array} 
$$
induces on $\overline{M}_{0,0}(LG(2,4),2)$ a structure of spherical variety. Furthermore, there exists an isomorphism 
$$\varphi : \mmle \rightarrow \mathcal{S}_4$$
where $\mathcal{S}_4$ is the wonderful compactification of the space of symplectic quadrics of $\mathbb{P}^3$, mapping a smooth conic $C\subset LG(2,4)$ to the quadric $\bigcup_{[L]\in C}L\subset\mathbb{P}^3$. The Cox ring $\Cox(\overline{M}_{0,0}(LG(2,4),2))$ is generated by the sections of $\Delta^2,D_{unb}^2,H_{\sigma_2}^2,T^2$. 

The moduli space $\mmle$ identifies with the blow-up of $\mathbb{G}(1,4)$ along the Veronese $\mathcal{V}_2^{3}$. With this identification the morphism associated to $H_{\sigma_2}^2$ is the blow-down and $\widetilde{Chow}(LG(2,4),2)\cong \mathbb{G}(1,4)$, while the morphism associated to $T^2$ is induced by the strict transform on $\mathcal{S}_4$ of the linear system of quadrics containing $\mathcal{V}_2^{3}$, and its image is a $6$-fold of degree $40$ in $\mathbb{P}^{14}$ isomorphic to $\overline{M}_{0,0}(LG(2,4),2,1)$.  

Finally, $\PsAut(\mmle)\cong\Aut(\mmle) \cong PSp(4)$ where $PSp(4)$ is the projective symplectic group, and $\PsAut(\mmle)$ is the group of birational self-maps of $\mmle$ inducing automorphisms in codimension one. 
\end{thm}

\subsection*{Organization of the paper} Throughout the paper we will work over an algebraically closed field $K$ of
characteristic zero. In Section \ref{sec1}, as a warm-up we prove some of the main results in \cite{Va82}, \cite{Va84}, using the techniques based on tangent cones computations that we will then apply to the more involved case of symplectic quadrics. In Section \ref{CSSF} we construct the wonderful compactification $\mathcal{S}_{2r}$ of the space of symmetric and symplectic $2r\times 2r$ matrices. In Section \ref{divS2r} we study the Picard rank, the effective and the nef cones of $\mathcal{S}_{2r}$. In Section \ref{birS2r} we compute the Mori chamber decomposition of the effective cone of $\mathcal{S}_{4}$ and $\mathcal{S}_{6}$. Finally, in Section \ref{K_Grass}, taking advantage of the theory of complete symplectic quadrics, we investigate the birational geometry of Kontsevich spaces of conics in Lagrangian Grassmannians. 

\subsection*{Acknowledgments}
We thank very much Alex Casarotti, Massimiliano Mella, Giorgio Ottaviani and Jason Starr for useful discussions, and the referee for many helpful comments that helped us to improve the exposition and correct a mistake about the sphericity of $\overline{M}_{0,0}(LG(r,2r),2)$ for $r > 2$ in a first version of the paper. 
 
The second named author is a member of the Gruppo Nazionale per le Strutture Algebriche, Geometriche e le loro Applicazioni of the Istituto Nazionale di Alta Matematica "F. Severi" (GNSAGA-INDAM). 

\section{Complete quadrics}\label{sec1}
Let $V$ be a $K$-vector space of dimension $n+1$, and let $\mathbb{P}^N$ with $N = \binom{n+2}{2}-1$ be the projective space parametrizing quadratic forms on $\mathbb{P}^n = \mathbb{P}(V)$ up to a scalar multiple. 

The line bundle $\mathcal{O}_{\mathbb{P}^n}(2)$ induces an embedding
$$
\begin{array}{cccc}
\nu:
&\mathbb{P}^n & \longrightarrow & \mathbb{P}^N\\
      & [x_0:\dots :x_n] & \longmapsto & [x_0^2:x_0x_1:\dots :x_n^2]
\end{array}
$$ 
The image $\mathcal{V}^n_2 = \nu(\mathbb{P}^n) \subset \mathbb{P}^{N}$ is the \textit{Veronese variety} of dimension $n$ and degree $2^n$. 
We will denote by $[z_{0,0}:\dots :z_{n,n}]$ the homogeneous coordinates on $\mathbb{P}^N$, where $z_{i,j}$ corresponds to the product $x_ix_j$.  

\subsubsection*{Secant varieties} 
Given an irreducible and reduced non-degenerate variety $X\subset\P^N$, and a positive integer $h\leq N$ we denote by $\sec_h(X)$ 
the \emph{$h$-secant variety} of $X$. This is the subvariety of $\P^N$ obtained as the closure of the union of all $(h-1)$-planes 
$\langle x_1,...,x_{h}\rangle$ spanned by $h$ general points of $X$. 

A point $p\in \mathbb{P}^N$ can be represented by an $(n+1)\times (n+1)$ symmetric matrix $Z$. The Veronese variety $\mathcal{V}^n_2$ is the locus of rank one matrices. More generally, $p\in \sec_h(\mathcal{V}^n_2)$ if and only if $Z$ can be written as a linear combination of $h$ rank one matrices that is if and only if $\rank(Z)\leq h$. If $p = [z_{0,0}:\cdots:z_{n,n}]$ then we may write
\stepcounter{thm}
\begin{equation}\label{matrix}
Z = \left(
\begin{array}{ccc}
z_{0,0} & \dots & z_{0,n}\\ 
\vdots & \ddots & \vdots\\ 
z_{0,n} & \dots & z_{n,n}
\end{array}\right)
\end{equation}
Then, the ideal of $\sec_h(\mathcal{V}^n_2)$ is generated by the $(h+1)\times (h+1)$ minors of $Z$. 

By \cite[Lemma 3.3]{Ma18a} the $SL(n+1)$-action
$$
\begin{array}{ccc}
SL(n+1)\times \mathbb{P}^n & \longrightarrow & \mathbb{P}^n\\
(M,[v]) & \longmapsto & [Mv]
\end{array}
$$
induces the $SL(n+1)$-action on $\mathbb{P}^{N}$ given by
\stepcounter{thm}
\begin{equation}\label{acss}
\begin{array}{ccc}
SL(n+1)\times \mathbb{P}^{N} & \longrightarrow & \mathbb{P}^{N}\\
(M,Z) & \longmapsto & MZM^{t}
\end{array}
\end{equation}

The orbit closures of the action (\ref{acss}) are precisely the secant varieties $\sec_h(\mathcal{V}^n_2)$. Now, let us recall the notion of spherical and wonderful variety.

\begin{Definition}
A \textit{spherical variety} is a normal variety $X$ together with an action of a connected reductive affine algebraic group $\mathscr{G}$, a Borel subgroup $\mathscr{B}\subset \mathscr{G}$, and a base point $x_0\in X$ such that the $\mathscr{B}$-orbit of $x_0$ in $X$ is a dense open subset of $X$. 

Let $(X,\mathscr{G},\mathscr{B},x_0)$ be a spherical variety. We distinguish two types of $\mathscr{B}$-invariant prime divisors: a \textit{boundary divisor} of $X$ is a $\mathscr{G}$-invariant prime divisor on $X$, a \textit{color} of $X$ is a $\mathscr{B}$-invariant prime divisor that is not $\mathscr{G}$-invariant. We will denote by $\mathcal{B}(X)$ and $\mathcal{C}(X)$ respectively the set of boundary divisors and colors of $X$.
\end{Definition}

For instance, any toric variety is a spherical variety with $\mathscr{B}=\mathscr{G}$ equal to the torus. For a toric variety there are no colors, and the boundary divisors are the usual toric invariant divisors.

\begin{Definition}
A \textit{wonderful variety} is a smooth projective variety $X$ with the action of a semi-simple simply connected group $\mathscr{G}$ such that:
\begin{itemize}
\item[-] there is a point $x_0\in X$ with open $\mathscr{G}$ orbit and such that the complement $X\setminus \mathscr{G}\cdot x_0$ is a union of prime divisors $E_1,\cdots, E_r$ having simple normal crossing;
\item[-] the closures of the $\mathscr{G}$-orbits in $X$ are the intersections $\bigcap_{i\in I}E_i$ where $I$ is a subset of $\{1,\dots, r\}$.
\end{itemize} 
\end{Definition}  

As proven by D. Luna in \cite{Lu96} wonderful varieties are in particular spherical. Note that $\mathbb{P}^N$ is not a wonderful compactification of $SL(n+1)/H$, where $H$ is the stabilizer of the identity matrix with respect to the $SL(n+1)$-action in (\ref{acss}), since for instance the orbit closure $\sec_n(\mathcal{V}^n_2)$ is a non smooth divisor. In order to get a wonderful compactification we must consider the space of complete quadrics that we now describe. The \textit{space of complete quadrics} is the closure of the graph of the rational map 
$$
\begin{array}{ccc}
\mathbb{P}(\Sym^2V)& \dasharrow & \mathbb{P}(\Sym^2\bigwedge^{2}V)\times\dots\times \mathbb{P}(\Sym^2\bigwedge^nV)\\
 Z & \longmapsto & (\wedge^2Z,\dots,\wedge^{n}Z)
\end{array}
$$

By \cite[Theorem 6.3]{Va82} the space of complete quadrics can be constructed as a sequence of blow-ups as follows.

\begin{Construction}\label{ccq}
Let us consider the following sequence of blow-ups:
\begin{itemize}
\item[-] $\mathcal{Q}(n)_1$ is the blow-up of $\mathcal{Q}(n)_0:=\mathbb{P}^{N}$ along the Veronese variety $\mathcal{V}^n_2$;
\item[-] $\mathcal{Q}(n)_2$ is the blow-up of $\mathcal{Q}(n)_1$ along the strict transform of $\sec_2(\mathcal{V}^n_2)$;\\
$\vdots$
\item[-] $\mathcal{Q}(n)_i$ is the blow-up of $\mathcal{Q}(n)_{i-1}$ along the strict transform of $\sec_i(\mathcal{V}^n_2)$;\\
$\vdots$
\item[-] $\mathcal{Q}(n)_{n-1}$ is the blow-up of $\mathcal{Q}(n)_{n-2}$ along the strict transform of $\sec_{n-1}(\mathcal{V}^n_2)$.
\end{itemize}
Let $f_i:\mathcal{Q}(n)_i\rightarrow \mathcal{Q}(n)_{i-1}$ be the blow-up morphism. We will denote by $E_i^q$ both the exceptional divisor of $f_i$ and its strict transforms in the subsequent blow-ups. We will denote by $\mathcal{Q}(n)$ the last blow-up $\mathcal{Q}(n)_{n-1}$ and by $f:\mathcal{Q}(n)\rightarrow\mathbb{P}^{N}$ the composition of the $f_i$. 

Then for any $i = 1,\dots,n-1$ the variety $\mathcal{Q}(n)_{i}$ is smooth, the strict transform of $\sec_{i+1}(\mathcal{V}^n_2)$ in $\mathcal{Q}(n)_{i}$ is smooth, and the divisor $E_1^q\cup E_2^q\cup\dots \cup E^q_{i}$ in $\mathcal{Q}(n)_{i}$ is simple normal crossing. Furthermore, the variety $\mathcal{Q}(n)$ is isomorphic to the space of complete $(n-1)$-dimensional quadrics.
\end{Construction}

In particular, $\mathcal{Q}(n)$ is a wonderful compactification of the homogeneous space $SL(n+1)/SO(n+1)$. We now recall some facts about the varieties $\sec_h(\mathcal{V}^n_2)$.

\begin{Remark}\label{sec_ver}
Recall that $\sec_h(\mathcal{V}^n_2)$ identifies with the variety parametrizing $(n+1)\times (n+1)$ symmetric matrices modulo scalar of rank at most $h$. An argument similar to the one used to estimate the dimension of the spaces of matrices, not necessarily symmetric, of rank at most $h$ in \cite[Example 12.1]{Ha95} shows that
$$\dim(\sec_h(\mathcal{V}^n_2)) = \frac{2nh-h^2+3h-2}{2}$$
for $h\leq n$. Furthermore, identifying $\sec_h(\mathcal{V}^n_2)$ with the variety parametrizing $(n+1)\times (n+1)$ symmetric matrices modulo scalar of corank at least $n+1-h$, by \cite[Proposition 12(b)]{HT84} we get that the degree of $\sec_h(\mathcal{V}^n_2)$ is given by
$$\deg(\sec_h(\mathcal{V}^n_2)) = \prod_{i=0}^{n-h}\frac{\binom{n+1+i}{n+1-h-i}}{\binom{2i+1}{i}}.$$
In particular, for $h = n$ we get $n+1$, and for $h = 1$ we get $2^n$.
\end{Remark}

\begin{Proposition}\label{tcones}
The tangent cone of $\sec_h(\mathcal{V}^n_2)$ at a point $p\in \sec_k(\mathcal{V}^n_2)\setminus\sec_{k-1}(\mathcal{V}^n_2)$ for $k\leq h$ is a cone with vertex of dimension $\binom{n+2}{2}-1-\frac{(n-k+1)(n-k+2)}{2}$ over $\sec_{h-k}(\mathcal{V}^{n-k}_2)$. In particular, for $k < h$ we have
$$\mult_{\sec_k(\mathcal{V}^n_2)\setminus\sec_{k-1}(\mathcal{V}^n_2)}\sec_h(\mathcal{V}^n_2) = \prod_{i=0}^{n-h}\frac{\binom{n-k+1+i}{n+1-h-i}}{\binom{2i+1}{i}}
$$ 
and $\Sing(\sec_h(\mathcal{V}^n_2)) = \sec_{h-1}(\mathcal{V}^n_2)$.
\end{Proposition}
\begin{proof}
We compute the tangent cone of $\sec_h(\mathcal{V}^n_2)$ at 
$$
p_k = \left(
\begin{array}{cc}
I_{k,k} & 0_{k,n+1-k} \\ 
0_{n+1-k,k} & 0_{n+1-k,n+1-k}
\end{array}
\right)
$$ 
where $I_{k,k}$ is the $k\times k$ identity matrix. Consider the affine chart $z_{0,0}\neq 0$ and the change of coordinates $z_{i,i}\mapsto z_{i,i}-1$ for $i = 1,\dots,k-1$, $z_{i,j}\mapsto z_{i,j}$ if $i\neq j$. Then the matrix $Z$ in (\ref{matrix}) takes the following form
$$
\left(
\begin{array}{ccccccc}
1 & z_{0,1} & \hdots & z_{0,k-1} & z_{0,k} & \hdots & z_{0,n} \\ 
z_{0,1} & z_{1,1}-1 & \hdots & z_{1,k-1} & z_{1,k} & \hdots & z_{1,n} \\ 
\vdots & \vdots & \ddots & \vdots & \vdots & \ddots & \vdots \\ 
z_{0,k-1} & z_{1,k-1} & \hdots & z_{k-1,k-1}-1 & z_{k-1,k} & \hdots & z_{k-1,n} \\ 
z_{0,k} & z_{1,k} & \hdots & z_{k-1,k} & z_{k,k} & \hdots & z_{k,n} \\ 
\vdots & \vdots & \ddots & \vdots & \vdots & \ddots & \vdots \\ 
z_{0,n} & z_{1,n} & \hdots & z_{k-1,n} & z_{k,n} & \hdots & z_{n,n}
\end{array} 
\right)
$$
Recall that $\sec_h(\mathcal{V}^n_2)\subseteq\mathbb{P}^N$ is cut out by the $(h+1)\times (h+1)$ minors of $Z$. Now, the lowest degree terms of these minors are given by the $(h+1-k)\times (h+1-k)$ minors of the following matrix
$$
\left(
\begin{array}{ccc}
z_{k,k} & \hdots & z_{k,n}\\ 
\vdots & \ddots & \vdots \\ 
z_{k,n} & \hdots & z_{n,n}
\end{array} 
\right)
$$
Therefore, the tangent cone $TC_{p_k}\sec_h(\mathcal{V}^n_2)$ is contained in the cone $C$ over $\sec_{h-k}(\mathcal{V}^{n-k}_2)$ with vertex the linear subspace of $\mathbb{P}^N$ given by $\{z_{k,k} =\dots = z_{k,n} = z_{k+1,k+1} = \dots = z_{k+1,n} = \dots = z_{n,n}=0\}$. Now, Remark \ref{sec_ver} yields 
$$\dim(C) = \binom{n+2}{2}-1-\frac{(n-k+1)(n-k+2)}{2}+\dim(\sec_{h-k}(\mathcal{V}^{n-k}_2))+1 = \dim(\sec_h(\mathcal{V}^n_2))$$
and hence $TC_{p_k}\sec_h(\mathcal{V}^n_2) = C$. Finally, to get the formula for the multiplicity it is enough to observe that 
$$\mult_{p_k}\sec_h(\mathcal{V}^n_2) = \mult_{p_k}TC_{p_k}\sec_h(\mathcal{V}^n_2) = \deg(\sec_{h-k}(\mathcal{V}^{n-k}_2))$$
and to apply the formula for the degree of the secant varieties of $\mathcal{V}^n_2$ in Remark \ref{sec_ver}.
\end{proof}

We will need the following result on fibrations with smooth fibers on a smooth base. 

\begin{Proposition}\label{smooth_fib}
Let $f:X\rightarrow Y$ be a surjective morphism of varieties over an algebraically closed field with equidimensional smooth fibers. If $Y$ is smooth then $X$ is smooth as well. 
\end{Proposition}
\begin{proof}
By \cite[Theorem 3.3.27]{Sch99} the morphism $f:X\rightarrow Y$ is flat. Finally, since all the fibers of $f:X\rightarrow Y$ are smooth and of the same dimension \cite[Theorem 3', Chapter III, Section 10]{Mum99} yields that $X$ is smooth. 

However, a direct proof is at hand and we present it in what follows. Since the problem is local on both $X$ and $Y$ we may assume that $X\subset K^N$ is an affine variety cut out by polynomials $g_1,\dots,g_a$, $Y = K^m$, and $f:X\rightarrow Y$ is given by $f(x) = (f_1(x),\dots, f_m(x))$. 

Consider a point $p\in X$. Without loss of generality we may assume that $f(p) = 0$. Then the fiber $X_0$ of $f$ through $p$ is given by
$$f^{-1}(0) = \{x\in K^N \: | \: g_1(x) = \dots = g_a(x) = f_1(x) = \dots = f_m(x) = 0\}.$$
Now, since $X_0$ is smooth at $p$ there are $b\leq a$ polynomials among $g_1,\dots,g_a$ and $l\leq m$ polynomials among $f_1,\dots,f_m$ such that $b+l = m+N-\dim(X)$ and the vectors
$$(\nabla g_1)(p),\dots,(\nabla g_b)(p),(\nabla f_1)(p),\dots,(\nabla f_l)(p)$$
are linearly independent. Now, $l\leq m$ yields $b\geq N-\dim(X)$. On the other hand, $X$ is irreducible of codimension $N-\dim(X)$ and hence $b\leq N-\dim(X)$. We conclude that $b = N - \dim(X)$ and the vectors
$$(\nabla g_1)(p),\dots,(\nabla g_{N - \dim(X)})(p)$$
are linearly independent. So $X$ is smooth at $p$.
\end{proof}

\begin{Notation}\label{Not_ST_Q}
We will denote by $\sec_h(\mathcal{V}^n_2)^i$ the strict transform of $\sec_h(\mathcal{V}^n_2)$ in $\mathcal{Q}(n)_i$ for $h > i$. Furthermore, as already said in Construction \ref{ccq}, for simplicity of notation we will denote by $E_i^q$ both the exceptional divisor of $f_i$ and its strict transforms in the subsequent blow-ups.
\end{Notation}

In the following we will analyze the geometry of the $SL(n+1)$-orbits in the blow-ups $\mathcal{Q}(n)_i$ in Construction \ref{ccq}. 
 
\begin{Proposition}\label{Qn_won}
For any $i = 0,\dots,n-1$ the variety $\mathcal{Q}(n)_i$ is smooth and the divisors $E^q_1,\dots,E^q_i$ are smooth and intersect transversally in $\mathcal{Q}(n)_i$. Furthermore, the strict transform $\sec_{i+1}(\mathcal{V}_2^n)^i$ of $\sec_{i+1}(\mathcal{V}_2^n)$ in $\mathcal{Q}(n)_i$ is smooth and the intersections among $\sec_{i+1}(\mathcal{V}_2^n)^i, E^q_1,\dots,E^q_i$ are transversal. The closures of the orbits of the $SL(n+1)$-action on $\mathcal{Q}(n)_i$ induced by (\ref{acss}) are given by all the possible intersections of $E^q_1,\dots,E^q_{i},\sec_{i+1}(\mathcal{V}_2^n)^i,\dots,\sec_{n}(\mathcal{V}_2^n)^i$ and $\mathcal{Q}(n)_i$ itself. 

In particular, the variety $\mathcal{Q}(n)$ is smooth, the divisors $E^q_1,\dots,E^q_{n-1},\sec_n(\mathcal{V}_2^n)^{n-1}$ are smooth and the intersections among them are transversal, the closures of the orbits of the $SL(n+1)$-action on $\mathcal{Q}(n)$ induced by (\ref{acss}) are given by all the possible intersections of the divisors $E^q_1,\dots,E^q_{n-1},\sec_n(\mathcal{V}_2^n)^{n-1}$ and $\mathcal{Q}(n)$ itself. Hence, $\mathcal{Q}(n)$ is wonderful.  
\end{Proposition}

\begin{proof}
We will proceed as follows. For $i = 0,1$ we will prove the statement for any $n$. Then we will prove that if for $i<j$ the statement holds for any $n$ then it also holds for $i = j$ and any $n$. This will prove the statement for any $n\geq 1$ and $i = 0,\dots, n-1$. 

For $i = 0$ we have $\mathcal{Q}(n)_0\cong \mathbb{P}^N$, there are no exceptional divisors, and the closures of the orbits of the action (\ref{acss}) are the secant varieties of $\mathcal{V}_2^n$. Therefore, for $i = 0$ the statements holds for any $n$. Even though we could use the case $i = 0$ as the first step of the proof, to get acquainted with the arguments we will apply, we develop in full detail the case $i = 1$ as well. 

The variety $\mathcal{Q}(n)_1$ is the blow-up of $\mathbb{P}^N$ along the Veronese variety $\mathcal{V}_2^n$. Hence it is smooth. By Proposition \ref{tcones} $\sec_2(\mathcal{V}_2^n)$ is smooth away from $\mathcal{V}_2^n$ and $\sec_2(\mathcal{V}_2^n)^1\cap E^q_1\rightarrow \mathcal{V}_2^n$ is a fibration whose fibers are isomorphic to $\mathcal{V}_2^{n-1}$. Hence, Proposition \ref{smooth_fib} yields that $\sec_2(\mathcal{V}_2^n)^1\cap E^q_1$ is smooth and since $\dim(\sec_2(\mathcal{V}_2^n)^1\cap E^q_1) = n+n-1 = 2n-1 = \dim(\sec_2(\mathcal{V}_2^n)^1)-1$ we conclude that $\sec_2(\mathcal{V}_2^n)^1$ is smooth and the intersection $\sec_2(\mathcal{V}_2^n)^1\cap E^q_1$ is transversal. 

Now, via the action of $SL(n+1)$ in (\ref{acss}) we can translate any fiber of $E^q_1$ over $\mathcal{V}_2^n$ to any other fiber. Fix one such fiber $E^q_{1,p}$. By Proposition \ref{tcones} we have that $\sec_h(\mathcal{V}_2^n)^1\cap E^q_{1,p} = \sec_{h-1}(\mathcal{V}_2^{n-1})$ and the action of $SL(n+1)$ in (\ref{acss}) restricts on $E^q_{1,p}$ to the corresponding action of $SL(n)$. This proves the statement about the orbits for $\mathcal{Q}(n)_1$ for any $n\geq 1$. 

Assume that for any $i < j$ the statement holds for any $n$. Since $\mathcal{Q}(n)_{j-1}$ and $\sec_{j}(\mathcal{V}^n_2)^{j-1}\subset\mathcal{Q}(n)_{j-1}$ are smooth the blow-up $\mathcal{Q}(n)_{j}$ of $\mathcal{Q}(n)_{j-1}$ along $\sec_{j}(\mathcal{V}^n_2)^{j-1}$ is smooth as well. Furthermore, since all the intersections among $\sec_{j}(\mathcal{V}_2^n)^{j-1}, E^q_1,\dots,E^q_{j-1}$ in $\mathcal{Q}(n)_{j-1}$ are transversal we have that all the intersections among $E^q_1,\dots,E^q_{j}$ in $\mathcal{Q}(n)_{j}$ are transversal as well. 

Now, consider an intersection of the following form $\sec_{j+1}(\mathcal{V}^n_2)^j\cap E^q_{j_1}\cap\dots \cap E^q_{j_t}$. By Proposition \ref{tcones} the restriction of the blow-down morphism
$$\sec_{j+1}(\mathcal{V}^n_2)^j\cap E^q_{j_1}\cap\dots \cap E^q_{j_t}\rightarrow E^q_{j_1}\cap \dots \cap E^q_{j_{t-1}}\cap \sec_{j_t}(\mathcal{V}^n_2)^{j_t-1}$$
has fibers isomorphic to $\sec_{j-j_t+1}(\mathcal{V}^{n-j_t})^{j-j_t}$. Since both $E^q_{j_1}\cap \dots \cap E^q_{j_{t-1}}\cap \sec_{j_t}(\mathcal{V}^n_2)^{j_t-1}$ and $\sec_{j-j_t-1}(\mathcal{V}^{n-j_t})^{j-j_t}$ are smooth Proposition \ref{smooth_fib} yields that $\sec_{j+1}(\mathcal{V}^n_2)^j\cap E^q_{j_1}\cap\dots \cap E^q_{j_t}$ is smooth as well. Moreover, note that 
$$
\dim(\sec_{j+1}(\mathcal{V}^n_2)^j\cap E^q_{j_1}\cap\dots \cap E^q_{j_t}) = \dim(E^q_{j_1}\cap \dots \cap E^q_{j_{t-1}}\cap \sec_{j_t}(\mathcal{V}^n_2)^{j_t-1})+ \dim(\sec_{j-j_t+1}(\mathcal{V}^{n-j_t})^{j-j_t})
$$
and 
$$\dim(E^q_{j_1}\cap \dots \cap E^q_{j_{t-1}}\cap \sec_{j_t}(\mathcal{V}^n_2)^{j_t-1}) = \frac{2nj_t-j_t^2+3j_t-2}{2}-(t-1)$$
yield that $\dim(\sec_{j+1}(\mathcal{V}^n_2)^j\cap E^q_{j_1}\cap\dots \cap E^q_{j_t})$ is given by
$$
\begin{array}{l}
\frac{2nj_t-j_t^2+3j_t-2}{2}-(t-1) + \frac{2(n-j_t)(j-j_t+1)-(j-j_t+1)^2+3(j-j_t+1)-2}{2} = \\ 
\frac{2n(j+1)-(j+1)^2+3(j+1)-2}{2}-t = \dim(\sec_{j+1}(\mathcal{V}^n_2)^j) - t
\end{array} 
$$
and hence the intersection $\sec_{j+1}(\mathcal{V}^n_2)^j\cap E^q_{j_1}\cap\dots \cap E^q_{j_t}$ is transversal.

In the following we prove the claim about the orbit closures. If an orbit closure in $\mathcal{Q}(n)_i$ is not contained in the exceptional divisor $E^q_i$ then it is the strict transform of an orbit closure in $\mathcal{Q}(n)_i$, and hence it is given as an intersection among $E^q_1,\dots, E^q_{i-1},\sec_{i+1}(\mathcal{V}^n_2)^i,\dots, \sec_{n}(\mathcal{V}^n_2)^i$. 

Now, let us analyze the orbit closures in the exceptional divisor $E^q_i$. The fibers of $E^q_i$ over $\sec_i(\mathcal{V}^n_2)^{i-1}$ are projective spaces of dimension 
$$N_{n-i} = \binom{n-i+2}{2}-1.$$
Moreover, $SL(n+1)$ acts transitively on fibers that lie over the same orbit in $\mathcal{Q}(n)_{i-1}$. Note that by Lemma \ref{tcones} $\sec_h(\mathcal{V}^n_2)^i$ intersects each of these $N_{n-i}$-dimensional projective spaces along $\sec_{h-i}(\mathcal{V}^{n-i}_2)$, and the $SL(n+1)$-action on $\mathcal{Q}(n)_{i}$ in (\ref{acss}) induces the corresponding $SL(n-i+1)$-action on the fibers of $E^q_i$. Finally, the statement on the orbit closures in $\mathcal{Q}(n)_{i-1}$ follows then from the statement on the orbit closures in $\mathcal{Q}(n-i)_{0}$.
\end{proof}

\section{Complete symmetric symplectic forms}\label{CSSF}
From now on we will consider the case $n+1 = 2r$ even. Let $\spn$ be the symplectic group of $2r\times 2r$ symplectic matrices, that is
$$\spn = \{M\in \Hom(V,V) \: | \: M^t \Omega M = \Omega \}$$
where 
\stepcounter{thm}
\begin{equation}\label{symform}
\Omega = \left(
\begin{array}{cc}
0 & I_{r,r} \\ 
-I_{r,r} & 0
\end{array}
\right) 
\end{equation}
is the standard symplectic form. Over an algebraically closed field of characteristic zero the symplectic group is a non-compact, irreducible, simply connected, simple Lie group. 

\begin{Remark}\label{tan_symp}
Let us write a $2r\times 2r$ matrix $M$ as 
$$
M = \left(
\begin{array}{cc}
A & B \\ 
C & D
\end{array} 
\right)
$$
where $A,B,C,D$ are four $r\times r$ matrices. The condition of being symplectic translates then into the following system of equations
$$
\left\lbrace
\begin{array}{lll}
-C^tA+A^tC & = & 0_{r,r};\\ 
-C^tB+A^tD  & = & I_{r,r};\\ 
-D^tA+B^tC  & = & -I_{r,r};\\ 
-D^tB+B^tD & = & 0_{r,r}.
\end{array}\right. 
$$
Considering the transformation
\stepcounter{thm}
\begin{equation}\label{trasl}
\left(\begin{array}{cc}
A & B\\ 
C & D
\end{array}\right)\mapsto
\left(\begin{array}{cc}
A-I_{r,r} & B\\ 
C & D-I_{r,r}
\end{array}\right)
\end{equation}
we get the following relations for the tangent space of $\spn$ at the identity
$$A = -D^t,\: B = B^t,\: C = C^t.$$
Hence, the tangent space of $\spn$ at the identity is the Lie algebra $\mathfrak{sp}(2r,K)$ consisting of $2r\times 2r$ matrices of the form
$$
\left(
\begin{array}{cc}
A & B \\ 
C & -A^t
\end{array} 
\right)
$$  
with $C$ and $B$ symmetric. In particular, $\dim(\spn) = r^2+2\frac{r(r+1)}{2} = r(2r+1)$.
\end{Remark}

\begin{Remark}\label{borelsym}
By \cite[Section 1]{ou} the Borel subgroup of the symplectic group can be described as follows:  
$$\mathscr{B}= \Bigr\{
\begin{pmatrix}
A & 0_{r,r}\\
B & A^{-t}
\end{pmatrix} 
\text{ with }  A^tB=B^tA\Bigl\}$$
\end{Remark}
where $A \in GL(r)$ is lower triangular and $B$ is a general $r \times r$  matrix. Now, $\spn$ is a subgroup of $SL(n+1)$ and the $SL(n+1)$-action (\ref{acss}) restricts to the following $\spn$-action:
\stepcounter{thm}
\begin{equation}\label{acsimp}
\begin{array}{ccc}
\spn\times \mathbb{P}^{N} & \longrightarrow & \mathbb{P}^{N}\\
(M,Z) & \longmapsto & MZM^{t}
\end{array}
\end{equation}
We denote by $O_{2r}$ the $\spn$-orbit of the identity in $\mathbb{P}^{N}$ and by $X_{2r} = \overline{O_{2r}}\subseteq\mathbb{P}^{N}$ its closure. 

\begin{Proposition}\label{orb_dim}
Let $Y_k = \overline{O_k}\subset\mathbb{P}^{N}$ be the closure of the $Sp(2r)$-orbit of the matrix
$$I_k = \left(\begin{array}{cc}
I_{k,k} & 0_{k,2r-k}\\ 
0_{2r-k,k} & 0_{2r-k,2r-k}
\end{array}\right)$$
via the action in (\ref{acss}). If $k\leq r$ then 
$$\dim(Y_k) = r(2r+1)-\frac{k(k-1)}{2}-r(r-k)-\frac{r(r+1)}{2}-\frac{(r-k)(r-k+1)}{2}-1=2rk+k-k^2-1.$$
Finally, $\dim(Y_{2r}) = r(r+1)$. 
\end{Proposition}
\begin{proof}
Our aim is to compute the dimension of the stabilizer $H\subset \spn$ of $I_k$. Consider the incidence correspondence 
\[
  \begin{tikzpicture}[xscale=1.5,yscale=-1.5]
    \node (A0_1) at (1, 0) {$\mathcal{I} = \{(M,\lambda)\: | \: MI_k M^t = \lambda I_k\}\subseteq \spn \times K^{*}$};
    \node (A1_0) at (0, 1) {$\spn$};
    \node (A1_2) at (2, 1) {$K^{*}$};
    \path (A0_1) edge [->]node [auto] {$\scriptstyle{\psi}$} (A1_2);
    \path (A0_1) edge [->]node [auto,swap] {$\scriptstyle{\phi}$} (A1_0);
  \end{tikzpicture}
\]
Note that the fibers of $\psi$ are isomorphic subgroups of $\spn$. We will compute the dimension of $H_1 = \psi^{-1}(1)$ and then the dimension of $H = \phi(\mathcal{I})$ will be given by
\stepcounter{thm}
\begin{equation}\label{dim_stab}
\dim(H) = \dim(\mathcal{I}) = \dim(H_1)+1.
\end{equation}
Consider first the case $k\leq r$. Subdivide as usual the matrices in $\spn$ in four $r\times r$ blocks and write the matrix whose orbit we want to study as 
$$\left(\begin{array}{cc}
Z_{k} & 0_{r,r}\\ 
0_{r,r} & 0_{r,r}
\end{array}\right)$$
where $Z_k$ is the following $r\times r$ matrix
$$Z_k = \left(\begin{array}{cc}
I_{k,k} & 0_{k,r-k}\\ 
0_{r-k,k} & 0_{r-k,r-k}
\end{array}\right).$$
Now, we have 
$$\left(\begin{array}{cc}
A & B\\ 
C & D
\end{array}\right)\left(\begin{array}{cc}
Z_{k} & 0_{r,r}\\ 
0_{r,r} & 0_{r,r}
\end{array}\right)
\left(\begin{array}{cc}
A^{t} & C^{t}\\ 
B^{t} & D^{t}
\end{array}\right) = 
\left(\begin{array}{cc}
AZ_{k}A^{t} & AZ_{k}C^{t}\\ 
CZ_{k}A^{t} & CZ_{k}C^{t}
\end{array}\right).
$$
Subdivide the matrix $A$ in blocks as follows
$$A = \left(\begin{array}{cc}
A_{k,k} & A_{k,r-k}\\ 
A_{r-k,k} & A_{r-k,r-k}
\end{array}\right).
$$
Then 
$$
\left(\begin{array}{cc}
A_{k,k} & A_{k,r-k}\\ 
A_{r-k,k} & A_{r-k,r-k}
\end{array}\right)
\left(\begin{array}{cc}
I_{k,k} & 0_{k,r-k}\\ 
0_{r-k,k} & 0_{r-k,r-k}
\end{array}\right)
\left(\begin{array}{cc}
A_{k,k}^t & A_{r-k,k}^t\\ 
A_{k,r-k}^t & A_{r-k,r-k}^t
\end{array}\right) = 
\left(\begin{array}{cc}
A_{k,k}A_{k,k}^t & A_{k,k}A_{r-k,k}^t\\ 
A_{r-k,k}A_{k,k}^t & A_{r-k,k}A_{r-k,k}^t
\end{array}\right).
$$
Therefore, considering the transformation (\ref{trasl}) we get the following relations for the tangent space of $H_1$ at the identity
$$A_{k,k} = -A_{k,k}^t, \: A_{r-k,k} = 0_{r-k,k}.$$
Moreover, subdividing $C$ as we did for $A$, we get that the matrix
$$
\left(\begin{array}{cc}
A_{k,k} & A_{k,r-k}\\ 
A_{r-k,k} & A_{r-k,r-k}
\end{array}\right)
\left(\begin{array}{cc}
I_{k,k} & 0_{k,r-k}\\ 
0_{r-k,k} & 0_{r-k,r-k}
\end{array}\right)
\left(\begin{array}{cc}
C_{k,k}^t & C_{r-k,k}^t\\ 
C_{k,r-k}^t & C_{r-k,r-k}^t
\end{array}\right) = 
\left(\begin{array}{cc}
A_{k,k}C_{k,k}^t & A_{k,k}C_{r-k,k}^t\\ 
A_{r-k,k}C_{k,k}^t & A_{r-k,k}C_{r-k,k}^t
\end{array}\right)
$$
must be zero. This yields the following further relations for the tangent space of $H_1$ at the identity
$$C_{k,k} = 0_{k,k},\: C_{r-k,k} = 0_{r-k,k}.$$
Plugging-in the relations for the tangent space at the identity of $\spn$ in Remark \ref{tan_symp} we get that the tangent space at the identity of $H_1$ is given by matrices of the form 
$$
\left(\begin{array}{cc}
A & B\\ 
C & -A^t
\end{array}\right)
$$
where $B = B^t$ and $C = C^t$. Note that $C = C^t, C_{k,k} = 0_{k,k},\: C_{r-k,k} = 0_{r-k,k}$ yield $C_{k,r-k} = 0_{k,r-k}$ and $C_{r-k,r-k} = C_{r-k,r-k}^t$.
Hence $A$ depends on $\frac{k(k-1)}{2}+k(r-k)+(r-k)^2$ parameters, $B$ depends on $\frac{r(r+1)}{2}$ parameters and $C$ depends on $\frac{(r-k)(r-k+1)}{2}$ parameters. Then by (\ref{dim_stab}) we get
$$\dim(H) = \frac{k(k-1)}{2}+k(r-k)+(r-k)^2+\frac{r(r+1)}{2}+\frac{(r-k)(r-k+1)}{2}+1$$  
and
$$\dim(Y_k) = \dim(\spn)-\dim(H) = r(2r+1)-\dim(H)$$
yields the formula in the statement. 

Finally, consider the case $k = 2r$, and let $H\subset \spn$ be the stabilizer of the identity matrix. The equality 
$$\left(\begin{array}{cc}
A & B\\ 
C & D
\end{array}\right)\left(
\begin{array}{cc}
I_{r,r} & 0_{r,r}\\ 
0_{r,r} & I_{r,r}
\end{array}\right)
\left(\begin{array}{cc}
A^{t} & C^{t}\\ 
B^{t} & D^{t}
\end{array}\right) = 
\left(\begin{array}{cc}
AA^{t}+BB^t & AC^{t}+BD^t\\ 
CA^{t}+DB^t & CC^{t}+DD^t
\end{array}\right) = 
\left(\begin{array}{cc}
\lambda I_{r,r} & 0_{r,r}\\ 
0_{r,r} & \lambda I_{r,r}
\end{array}\right)
$$
for some $\lambda\in K^{*}$ yields, applying as usual the transformation (\ref{trasl}), the following system of equations
$$
\left\lbrace
\begin{array}{lll}
-C^t(A-I_{r,r})+(A-I_{r,r})^tC & = & 0_{r,r};\\ 
-C^tB^t+(A-I_{r,r})^t(D+I_{r,r}) & = & (D-I_{r,r})^t(A-I_{r,r})-B^tC;\\ 
-(D-I_{r,r})^tB+B^t(D-I_{r,r}) & = & 0_{r,r}.
\end{array}\right. 
$$
Note that if $M\in H$ taking the determinants on both sides of $M^T\Omega M = \lambda\Omega$ we see that $\lambda$ can only take finitely many values. Hence, by Remark \ref{tan_symp} we have the following relations for the tangent space of $H$ at the identity
$$A = -D^t,\: B = B^t,\: C = C^t,\: C = -B^t,\: A = -A^t.$$
Therefore, the tangent space consists of matrices of the following form
$$
\left(\begin{array}{cc}
A & B\\ 
-B^t & -A^t
\end{array}\right)
$$
with $B = B^t$ and $A = -A^t$. We conclude that 
$$\dim(H) = \frac{r(r+1)}{2}+\frac{r(r-1)}{2} = r^2$$
and hence $\dim(Y_{2r}) = r(2r+1)-r^2 = r(r+1)$.
\end{proof}

\begin{Corollary}\label{prop_dim}
The projective variety $X_{2r}$ is irreducible and its dimension is given by $\dim(X_{2r}) = r(r+1)$.
\end{Corollary}
\begin{proof}
The variety $X_{2r}$ is the closure of an $\spn$-orbit, so it is irreducible. Since $X_{2r} = Y_{2r}$ the formula for its dimension follows from Proposition \ref{orb_dim}. 
\end{proof}

\begin{Example}
Consider the case $r = 1$. Then Corollary \ref{prop_dim} yields $\dim(X_{2}) = 2$ and hence $X_{2} = \mathbb{P}^2$. Moreover, $O_{2} = \mathbb{P}^2\setminus C$ where $C\subset\mathbb{P}^2$ is the conic parametrizing rank one matrices. 
\end{Example}

\begin{Remark}\label{equations}
We work out equations for $X_{2r}$. The points of the orbit $O_{2r}$ represent symmetric matrices having a scalar multiple that is symplectic, that is $Z^t \Omega Z = \lambda \Omega$ for some $\lambda\in K^{*}$. The matrix $N =  Z^t \Omega Z$ is skew-symmetric and so $N_{i,i} = 0$ for $i = 0,\dots, 2r-1$. Furthermore, for any $i = 0,\dots,2r-2$ we must have
$$N_{i,i+1} = \dots = N_{i,r+i-1} = N_{i,r+i+1} = \dots = N_{i,2r-1} = 0.$$
This gives $2r-i-2$ quadratic equations for any $i = 0,\dots,r-1$, and $2r-i-1$ quadratic equations for any $i = r,\dots, 2r-1$. Moreover, we must have
$$N_{0,r} = N_{1,r+1} = \dots = N_{r-1,2r-1}$$
and hence we get $r-1$ additional quadratic equations. Summing-up we get 
$$\sum_{i=0}^{r-1} (2r-i-2) + \sum_{i = r}^{2r-1}(2r-i-1)+ r-1 = (2r+1)(r-1)$$
quadratic equations for $X_{2r}$ in $\mathbb{P}^{N}$.

Now, we explicitly compute these equations. Consider a general symmetric matrix $Z=(z_{i,j})_{i,j=0,\dots,2r-1}$ with $z_{i,j}=z_{j,i}$ and the standard symplectic form 
$\Omega$. Then
\begin{align*}
c_{i,j} := (Z \cdot \Omega)_{i,j} = \sum_{k=0}^{2r-1} z_{i,k} \Omega_{k,j} = \begin{cases} z_{i,j-r} & \text{ for } j \ge r; \\ -z_{i,j+r} & \text{ for } j < r; \\ \end{cases}
\end{align*}
and so
\begin{align*}
N_{i,j}:=(Z \cdot \Omega \cdot Z)_{i,j} &= \sum_{k=0}^{2r-1} c_{i,k} z_{k,j} = \sum_{k=0}^{r-1} c_{i,k} z_{k,j} +  \sum_{k=r}^{2r-1} c_{i,k} z_{k,j} = \sum_{k=0}^{r-1} -z_{i,k+r} z_{k,j} +  z_{i,k} z_{k+r,j}.
\end{align*}
Summing-up, the equations
$$\begin{cases}
N_{l,r+l}-N_{l+1,r+l+1}=0 &\text{ for } l=0, \dots, r-2;\\
N_{i,j}=0 &\text{ for } i=0, \dots, 2r-2 , \,  j > i, \, j \neq r+i;\\
\end{cases}$$
can be explicitly written as follow
\begin{equation*}
\begin{footnotesize}
\begin{cases}
\sum_{k=0}^{r-1}-z_{l,k+r}z_{k,r+l}+z_{l,k}z_{k+r,r+l}+z_{l+1,k+r}z_{k,r+l+1}-z_{l+1,k}z_{k+r,r+l+1}=0 &\text{ for } l=0, \dots, r-2;\\
\sum_{k=0}^{r-1} -z_{i,k+r} z_{k,j} +  z_{i,k} z_{k+r,j}=0  &\text{ for } i=0, \dots, 2r-2, \, j > i, \, j \neq r+i.\\
\end{cases}
\end{footnotesize}
\end{equation*}
\end{Remark}

Now, our aim is to construct a wonderful compactification of the space of complete symmetric symplectic forms. 

\begin{Construction}\label{ccssf}
Set $S_h(\vr):=\sec_{h}(\mathcal{V}^{2r-1}_2) \cap \xn$. Let us consider the following sequence of blow-ups:
\begin{itemize}
\item[-] $X_{2r}^{(1)}$ is the blow-up of $X_{2r}^{(0)}:=X_{2r}$ along the Veronese variety $\mathcal{V}^{2r-1}_2 \subset \xn$;
\item[-] $X_{2r}^{(2)}$ is the blow-up of $X_{2r}^{(1)}$ along the strict transform of $S_2(\vr)$;\\
$\vdots$
\item[-] $X_{2r}^{(i)}$ is the blow-up of $X_{2r}^{(i-1)}$ along the strict transform of $S_i(\vr)$;\\
$\vdots$
\item[-] $X_{2r}^{(r-1)}$ is the blow-up of $X_{2r}^{(r-2)}$ along the strict transform of $S_{r-1}(\vr)$.
\end{itemize}
Let $f_i:X_{2r}^{(i)}\rightarrow X_{2r}^{(i-1)}$ be the blow-up morphism. We will denote by $E_i$ both the exceptional divisor of $f_i$ and its strict transforms in the subsequent blow-ups. We set $\xnb:=X_{2r}^{(r-1)}$ and we indicate with $f: \xnb \rightarrow \xn$ the composition of the $f_i$. 
\end{Construction}

Let $M_{2r,2r}(K)$ be the space of $2r \times 2r$ matrices. Following \cite{defa} we define the operator
$$
\begin{array}{cccc}
\Phi_\Omega: & M_{2r,2r}(K) & \longrightarrow & M_{2r,2r}(K)\\
 & A & \mapsto & \Omega^{-1}A^T\Omega
\end{array}
$$
\begin{Definition}
A matrix $A \in \mm$ is symplectically congruent to a matrix $B \in \mm$ if there exists a symplectic matrix $Q$ such that $QAQ^T=B$.
\end{Definition}

By \cite[Theorem 21]{defa} a matrix $A \in \mm$ is symplectically congruent to a diagonal matrix if and only if $A$ is symmetric and $A\Phi_\Omega(A)$ is diagonalizable.

\begin{Proposition}\label{fun}
The quadratic equations in Remark \ref{equations} cut out $X_{2r}$ set-theoretically. Furthermore $Y_i= \sec_{i}(\vr) \cap \xn$, set-theoretically, and there is a stratification 
$$Y_{1}\subset Y_2\subset\dots\subset Y_{r-1}\subset Y_{r}\subset Y_{2r} = X_{2r}.$$ 
In particular, $\dim(\sec_{i}(\vr) \cap \xn) = 2ri+i-i^2-1$ for $i = 1,\dots,r$ and $Y_{r}$ is a divisor in $X_{2r}$.
\end{Proposition}
\begin{proof}
Let $Z$ be a symmetric matrix satisfying the equations in Remark \ref{equations}. Then we have two cases:
\begin{itemize}
\item[(i)] $N_{0,r} = \dots = N_{r-1,2r-1} = \lambda\in K^{*}$;
\item[(ii)] $N_{0,r} = \dots = N_{r-1,2r-1} = 0$.
\end{itemize}
Consider (i). Then $Z^{t}\Omega Z = \lambda\Omega$ and $\det(Z) \neq 0$. Moreover,
$$Z\Phi_\Omega(Z) = Z\Omega^{-1}Z^{t}\Omega = -Z^{t}\Omega Z\Omega = -\lambda\Omega^2 = \lambda I_{2r,2r}$$
and by \cite[Theorem 21]{defa} $Z$ is symplectically congruent to a diagonal matrix.

In case (ii) $Z^{t}\Omega Z$ is the zero matrix. So $\det(Z) = 0$, and $Z\Phi_\Omega(Z)$ is the zero matrix as well. Again, \cite[Theorem 21]{defa} yields that $Z$ is symplectically congruent to a diagonal matrix.

So if $Z$ is a symmetric matrix satisfying the equations in Remark \ref{equations} there is a symplectic matrix $Q$ such that $QZQ^{t} = D$ with $D$ diagonal. Our aim is to prove that $D$ can be moved to a matrix of the form $I_k$, where $k$ is the rank of $D$, with the action of the symplectic group. 

Let $D_{\alpha} = \diag(\alpha_1,\dots,\alpha_{2r})$ be a diagonal matrix satisfying the equations in Remark \ref{equations}. Then either $\alpha_i\alpha_{r+i} = 0$ for all $i = 1,\dots,r$, or $\alpha_i  \alpha_{r+i} = \lambda\in K^{*}$ for all $i = 1,\dots,r$. Write
$$
D_{\alpha} = \left(
\begin{array}{cc}
D_{\alpha_1,\dots,\alpha_p} & 0_{r,r} \\ 
0_{r,r} & D_{\alpha_{p+1},\dots,\alpha_{p+q}}
\end{array} \right)
$$
with $p+q\leq r$, where $D_{\alpha_1,\dots,\alpha_p}$ is an $r\times r$ diagonal matrix with the $\alpha_i$ appearing on the diagonal, and similarly for $D_{\alpha_{p+1},\dots,\alpha_{p+q}}$. Note that up to permuting the upper and lower diagonal simultaneously we may assume that $\alpha_1,\dots,\alpha_p$ are the first $p$ entries on the diagonal of $D_{\alpha_1,\dots,\alpha_p}$, and $\alpha_{p+1},\dots,\alpha_{p+q}$ are the last $q$ entries on the diagonal of $D_{\alpha_{p+1},\dots,\alpha_{p+q}}$. 

Now, set $p+q = r$. Let $A,B,C,D$ be $r\times r$ matrices defined as follows:
\begin{itemize}
\item[-] the first $p$ entries on the diagonal of $A$ are $a_i \in K^{*}$ for $i = 1,\dots,p$, and the other entries are zero;
\item[-] the last $q$ entries on the diagonal of $B$ are $-b_i^{-1} \in K^{*}$ for $i = p+1,\dots,p+q$, and the other entries are zero;
\item[-] the last $q$ entries on the diagonal of $C$ are $b_i\in K^{*}$ for $i = p+1,\dots,p+q$, and the other entries are zero;
\item[-] the first $p$ entries on the diagonal of $D$ are $a_i^{-1}\in K^{*}$ for $i = 1,\dots,p$, and the other entries are zero.
\end{itemize}
Consider the matrix
$$
P = \left(
\begin{array}{cc}
A & B \\ 
C & D
\end{array}\right) 
$$
and note that $P$ is symplectic. Furthermore, by taking $a_i,b_j$ such that $a_i^2 = \alpha_i$ for $i = 1,\dots,p$, and $b_j^{-2} = \alpha_j$ for $j = p+1,\dots,p+q$ we have $P^{t}I_{r}P = D_{\alpha}$ when $p+q = r$. 

If $p+q < r$ by permuting the upper diagonal of $I_{p+q}$, we transform $I_{p+q}$ into the matrix $I_{p+q}^{*}$ whose entries on the diagonal are $(I_{p+q}^{*})_{i,i} = 1$ for $i = 1,\dots, p$, $(I_{p+q}^{*})_{i,i} = 0$ for $i = p+1,\dots, p+s$, $(I_{p+q}^{*})_{i,i} = 1$ for $i = p+s+1,\dots, p+s+q$, and $(I_{p+q}^{*})_{i,i} = 0$ for $i = p+s+q+1,\dots, 2r$, where $p+s+q = r$. In this case consider $r\times r$ diagonal matrices $\overline{A},\overline{B},\overline{C},\overline{D}$ such that
\begin{itemize}
\item[-] the first $p$ entries on the diagonal of $\overline{A}$ are $a_i \in K^{*}$ for $i = 1,\dots,p$, $(\overline{A})_{i,i} = 1$ for $i = p+1,\dots,p+s$, $(\overline{A})_{i,i} = 0$ for $i = p+s+1,\dots,p+s+q$;
\item[-] the first $p+s$ entries on the diagonal of $\overline{B}$ are zero, followed by $-b_{p+1}^{-1},\dots,-b_{p+q}^{-1}$;
\item[-] the first $p+s$ entries on the diagonal of $\overline{C}$ are zero, followed by $b_{p+1},\dots,b_{p+q}$;
\item[-] the first $p$ entries on the diagonal of $\overline{D}$ are $a_i^{-1} \in K^{*}$ for $i = 1,\dots,p$, $(\overline{D})_{i,i} = 1$ for $i = p+1,\dots,p+s$, $(\overline{D})_{i,i} = 0$ for $i = p+s+1,\dots,p+s+q$;
\end{itemize}
and set 
$$
\overline{P} = \left(
\begin{array}{cc}
\overline{A} & \overline{B} \\ 
\overline{C} & \overline{D}
\end{array}\right). 
$$
Then $\overline{P}$ is symplectic and by taking again $a_i,b_j$ such that $a_i^2 = \alpha_i$ for $i = 1,\dots,p$, and $b_j^{-2} = \alpha_j$ for $j = p+1,\dots,p+q$ it holds $\overline{P}^{t}I_{p+q}^{*}\overline{P} = D_{\alpha}$.

Furthermore, when $D_{\alpha}$ is of maximal rank we consider the diagonal symplectic matrix 
$$P = \diag(a_{1},\dots,a_r,a_{1}^{-1},\dots,a_{r}^{-1}).$$
Note that taking $a_{i}\in K^{*}$ such that $a_i^2 = \alpha_i\mu^{-1}$, with $\mu^2 = \lambda$, for $i = 1,\dots,r$, we get that $P^{t}I_{2r,2r}P$ is a scalar multiple of $D_{\alpha}$. Consider the matrices
$$
\Psi_t = \left(
\begin{array}{cc}
I_{r,r} & 0_{r,r} \\ 
0_{r,r} & T_{r,r}
\end{array} 
\right);\quad
\Lambda_t = \left(
\begin{array}{ccc}
I_{k,k} & 0 & 0_{k,2r-k-1}\\ 
0 & t & 0_{1,2r-k-1}\\ 
0_{2r-k-1,k} & 0_{2r-k-1,1} & 0_{2r-k-1,2r-k-1}
\end{array} \right) \quad \text{for}\: k = 1,\dots,r-1
$$
where $T_{r,r} = \diag(t,\dots,t)$. By the first part of the proof we have that $\{\Psi_t\}_{t\in K^{*}}$ is a family of matrices in $O_{2r}$, and $\lim_{t\mapsto 0}\Psi_t = I_{r}$. Furthermore, $\{\Lambda_t\}_{t\in K^{*}}$ is a family of matrices in $O_{k+1}$, and $\lim_{t\mapsto 0}\Lambda_{t} = I_{k}$ for $k = 1,\dots,r-1$.

Summing-up we proved that if $Z$ is a symmetric matrix for rank $k$ with $1\leq k\leq r$ or $k = 2r$ satisfying the equations in Remark \ref{equations} then $Z$ can be symplectically transformed into the matrix $I_k$, and hence it lies in $O_k$.
\end{proof}

\begin{Remark}\label{Ver_c}
Proposition \ref{fun} yields that the Veronese variety $\vr$ is contained in $\xn$. On the other hand, for $h\geq 2$ the secant variety $\sec_h(\vr)$ is not contained in $\xn$.
\end{Remark}

\begin{Proposition}\label{sec}
For any $k \ge r$ we have that $\xn \cap \sec_{r}(\vr) = \xn \cap \sec_{k}(\vr)$
\end{Proposition}
\begin{proof}
Assume there is a matrix $M \in \xn \cap \sec_{k}(\vr)$ of rank $r< k < 2r-1$. Arguing as in the proof of Proposition \ref{fun} we can move $M$ with the action of $\spn$ in a diagonal matrix $D_k$ of rank $k$, and $D_{k}\in \xn \cap \sec_{k}(\vr)$. However, $D_{k}$ does not satisfy the equation $N_{j,r+j}-N_{j+1, r+j+1}=0$ for $X_{2r}$ in Remark \ref{equations}. A contradiction. 
\end{proof}

We analyze in detail the geometry of the objects we introduced in the first non trivial case, namely $r = 2$.

\begin{Proposition}\label{x4g14}
The variety $X_4$ is isomorphic to the Grassmannian $\mathbb{G}(1,4)\subset\mathbb{P}^9$ of lines in $\mathbb{P}^4$. Furthermore, $\mathcal{V}_2^{3}\subset X_4$, and $S_2(\mathcal{V}_2^{3})\subset X_4$ is an irreducible and reduced divisor singular along $\mathcal{V}_2^{3}$. In particular, the equations in Remark \ref{equations} cut out $X_4$ scheme-theoretically, and $S_2(\mathcal{V}_2^{3}) = Y_2$ scheme-theoretically.  
\end{Proposition}
\begin{proof}
Consider homogeneous coordinates $z_{i,j}$ on $\mathbb{P}^9$ and identify them with the entries of a general $4\times 4$ symmetric matrix $Z$. The change of variables $z_{0,0}\mapsto z_{1,2}, z_{0,1}\mapsto z_{0,1}, z_{0,2}\mapsto \frac{z_{1,4}-z_{2,3}}{2}, z_{0,3}\mapsto z_{0,2}, z_{1,1}\mapsto z_{1,3}, z_{1,2}\mapsto z_{0,3}, z_{1,3}\mapsto \frac{z_{1,4}+z_{2,3}}{2}, z_{2,2}\mapsto z_{3,4}, z_{2,3}\mapsto z_{0,4}, z_{3,3}\mapsto z_{2,4}$ transforms the five equations in Remark \ref{equations} into the standard Pl\"ucker equations cutting out $\mathbb{G}(1,4)$ in $\mathbb{P}^9$. 

By Remark \ref{Ver_c} we have $\mathcal{V}_2^{3}\subset X_4$. We can compute the tangent cones of $S_2(\mathcal{V}_2^{3})\subset X_4$ at a point representing a rank one matrix, and at a point representing a rank two matrix using the equations for $X_{4}$ in Remark \ref{equations} together with the equations cutting out $\sec_{2}(\mathcal{V}_2^{3})$. In the first case we get a cone with a $3$-dimensional vertex over $\mathcal{V}_1^{2}$ which in particular is irreducible and reduced, and in the second case we get a $5$-dimensional linear space. Finally, since by Proposition \ref{fun} $Y_2$ has dimension five we conclude that the equations in Remark \ref{equations} together with the equations cutting out $\sec_{2}(\mathcal{V}_2^{3})$ define $Y_2$ scheme-theoretically. 
\end{proof}

\begin{Remark}
The variety $X_4\cong\mathbb{G}(1,4)$ has been studied in relation to moduli spaces of rank two vector bundles over a smooth quadric \cite[Table I]{OS94}.
\end{Remark}

\begin{Proposition}\label{symplcones}
The tangent cone $TC_{p_k}(\xn)$ of $\xn$ at a point $p_k\in S_k(\vr) \setminus S_{k-1}(\vr)$ for $k=1, \dots, r-1$ is a cone with vertex of dimension $k(2r+1-k)-1$ over $X_{2(r-k)}$. Moreover, $X_{2r}$ is smooth along $S_{r}(\vr) \setminus S_{r-2}(\vr)$, and the equations in Remark \ref{equations} define $X_{2r}$ scheme-theoretically.  

The tangent cone $TC_{p_k}(S_{h}(\vr))$ of $S_h(\vr)$ at a point $p_k\in S_k(\vr) \setminus S_{k-1}(\vr)$ for $k=1, \dots, r-1, k< h$ is a cone with vertex of dimension $k(2r+1-k)-1$ over $S_{h-k}(\mathcal{V}_2^{2(r-k)-1})$. Moreover, the equations in Remark \ref{equations} together with the equations cutting out $\sec_{h}(\vr)$ define $S_h(\vr)$ scheme-theoretically.  

In particular, $\xn$ is smooth along $S_{r}(\vr) \setminus S_{r-2}(\vr)$ and $S_{r}(\vr)$ is a divisor in $\xn$.
\end{Proposition}
\begin{proof}
Let $p_{k}=(p_{i,j})_{i,j=0, \dots, 2r-1, i \le j}$ be the point representing the standard matrix of rank $k$ with $p_{i,i}=1$ for $i=0,\dots,k-1$ and $p_{i,j}=0$ otherwise.   

We proceed by induction on $r$. The base case $r = 2$ is in Proposition \ref{x4g14}. We will use the equations in Remark \ref{equations} to compute $TC_{p_k}X_{2r}$. Consider the change of coordinates $z_{i,i}\mapsto z_{i,i}-z_{0,0}$ for $i = 1,\dots,k-1$, and set $z_{0,0} = 1$. Note that the lowest degree terms of the equations in Remark \ref{equations} after this change of coordinates are obtained by removing from $Z^t\Omega Z = \lambda\Omega$ the rows and columns indexed by $0,\dots,k-1,r,\dots,r+k-1$. Therefore, we get a cone with vertex of dimension $k(2r+1-k)-1$ over $X_{2(r-k)}$ which by induction hypothesis is irreducible and reduced since the equations in Remark \ref{equations} define $X_{2(r-k)}$ scheme-theoretically. Now, $k(2r+1-k)+\dim(X_{2(r-k)}) = \dim(X_{2r})$ yields that this is the tangent cone $TC_{p_k}X_{2r}$, and hence the equations in Remark \ref{equations} define $X_{2r}$ scheme-theoretically. Note that at the points representing $I_r$ and $I_{2r,2r}$ the equations in Remark \ref{equations} yield a linear subspace of the same dimension of $X_{2r}$.

Now, consider $S_h(\mathcal{V}_2^{2r-1})$. Note that $TC_{p_k}S_h(\mathcal{V}_2^{2r-1})$ is contained in $TC_{p_k}X_{2r}\cap TC_{p_k}\sec_h(\mathcal{V}_2^{2r-1})$. By the previous computation of $TC_{p_k}X_{2r}$ and the computation of $TC_{p_k}\sec_h(\mathcal{V}_2^{2r-1})$ in Proposition \ref{tcones}, we conclude that $TC_{p_k}X_{2r}\cap TC_{p_k}\sec_h(\mathcal{V}_2^{2r-1})$ is a cone with vertex of dimension $k(2r+1-k)-1$ over $S_h(\mathcal{V}_2^{2(r-k)-1}) = \sec_{h-k}(\mathcal{V}_2^{2(r-k)-1})\cap X_{2(r-k)}$. Again by induction this is an irreducible and reduced cone which  by the computation of the dimension of $S_h(\mathcal{V}_2^{2r-1})$ in Proposition \ref{fun} must coincide with $TC_{p_k}S_h$. Hence the equations in Remark \ref{equations} together with the equations cutting out $\sec_{h}(\vr)$ define $S_h(\vr)$ scheme-theoretically.  
\end{proof} 

Now, we are ready to prove the main result of this section. We will denote by $S_h^{(i)}(\vr)$ the strict transform of $S_h(\vr)$ in $X_{2r}^{(i)}$. 

\begin{thm}\label{main1}
For any $i = 1,\dots,r-1$ the strict transform $S_{i+1}^{(i)}(\mathcal{V}_2^{2r-1})$ of $S_{i+1}(\mathcal{V}_2^{2r-1})$ in $X_{2r}^{(i)}$ is smooth. Moreover, the variety $\xnb$ is smooth and the exceptional divisors $E_1, \dots, E_{r-1} \subset \xnb$ are smooth as well. 

The closures of the orbits of the $Sp(2r)$-action on $\xnb$ induced by the action in (\ref{acsimp}) are given by all the possible intersections among $E_1,\dots,E_{r-1},S_{r}^{(r-1)}(\mathcal{V}_2^{2r-1})$ and $X_{2r}^{(i)}$ itself.

In particular, the variety $\xnb$ with boundary divisors $E_1,\dots,E_{r-1},S_{r}^{(r-1)}(\mathcal{V}_2^{2r-1})$ is wonderful.
\end{thm}

\begin{proof} 
For every $r$ in $X_{2r}^{(0)}$ we have $S_1^{(0)}(\vr)=\vr$ which is smooth. We will assume that $S_j^{(j-1)}(\vr)$ is smooth for every $r$ and for every $j <i$ and prove that also $S_i^{(i-1)}(\vr)$ in $X_{2r}^{(i-1)}$ is smooth. We have $S_{i}^{(i-1)}(\vr) = \sec_{i}^{(i-1)}(\mathcal{V}_2^{2r-1}) \cap X_{2r}^{(i-1)}$, so we consider $S_{i}^{(i-1)}(\vr)$ inside $\mathcal{Q}(2r-1)_{(i-1)}$. By Proposition \ref{Qn_won}, for every $r$ and for every $i=0, \dots, 2r-1$ the varieties $\mathcal{Q}(2r-1)_{(i-1)}$, $\sec_{i}^{(i-1)}(\mathcal{V}_2^{2r-1})$, $E_1^q, \dots,E_{i-1}^q$ are smooth. 

Now, $S_{i}^{(i-1)}(\vr)$ is smooth away from $E_1^{q}, \dots,E_{i-1}^{q}$. Moreover, by Proposition \ref{symplcones} for every $k=1, \dots, i-1$,  $S_{i}^{(i-1)}(\vr) \cap E_k^{q} \rightarrow S_{k}^{(k-1)}(\vr)$ is a fibration with fibers isomorphic to $S_{i-k}^{(i-k-1)}(\mathcal{V}_2^{2(r-k)-1})$ which is smooth by induction. Proposition \ref{smooth_fib} yields that $S_{i}^{(i-1)}(\vr) \cap E_k^{q}$ is smooth for $k=1, \dots, i-1$. Now, since by Proposition \ref{fun} we have
$$\dim S_k^{(k-1)}(\mathcal{V}^{2r-1}_2)+\dim S_{i-k}^{(i-k-1)}(\mathcal{V}^{2(r-k)-1}_2) = 2ri+i-i^2-2 = \dim S_i^{(i-1)}(\mathcal{V}^{2r-1}_2)-1$$   
we get that $S_{i}^{(i-1)}(\vr)$ is smooth as well.

By Proposition \ref{symplcones} for every $r$, in $X_{2r}^{(1)}$ we have that $E_1 \cap S_2^{(1)}(\vr) \rightarrow \vr$ is a fibration with fibers isomorphic to $\mathcal{V}_2^{2(r-1)-1}$ and then by Proposition \ref{smooth_fib} $E_1 \cap S_2^{(1)}(\vr)$ is smooth of dimension $4r-4 = \dim(S_2^{(1)}(\vr))-1$. More generally, consider intersections of the form  $S_{i+1}^{(i)}(\vr)\cap E_{j_1}\cap\dots \cap E_{j_t}$, for $1 \le j_1 < \dots < j_t \le i$. By Proposition \ref{symplcones}, the restriction of the blow-down morphism
$$S_{i+1}^{(i)}(\vr) \cap E_{j_1} \cap \dots \cap E_{j_t}\rightarrow E_{j_1}\cap \dots \cap E_{j_{t-1}}\cap S_{j_t}^{(j_t-1)}(\vr)$$
has fibers isomorphic to $S_{i+1-j_t}^{(i-j_t)}(\mathcal{V}_2^{2(r-j_t)-1})$. Now, by Proposition \ref{smooth_fib} $S_{i+1}^{(i)}(\vr)\cap E_{j_1}\cap\dots \cap E_{j_t}$ is smooth since we proved before that $S_{i+1-j_t}^{(i-j_t)}(\mathcal{V}_2^{2(r-j_t)-1})$ is smooth and $E_{j_1}\cap \dots \cap E_{j_{t-1}}\cap S_{j_t}^{(j_t-1)}(\vr)$ is smooth by induction. Moreover,
$$
\dim(S_{i+1}^{(i)}(\vr) \cap E_{j_1} \cap \dots \cap E_{j_t}) = \dim(E_{j_1}\cap \dots \cap E_{j_{t-1}}\cap S_{j_t}^{(j_t-1)}(\vr))+ \dim(S_{i+1-j_t}^{(i-j_t)}(\mathcal{V}_2^{2(r-j_t)-1}))
$$
and by induction $\dim(E_{j_1}\cap \dots \cap E_{j_{t-1}}\cap S_{j_t}^{(j_t-1)}(\vr)) = 2rj_t+j_t-j_t^2-1-(t-1)$.
This yields, using Proposition \ref{fun}, that
\stepcounter{thm}
\begin{equation}\label{eq_dim}
\begin{array}{l}
\dim(S_{i+1}^{(i)}(\vr) \cap E_{j_1} \cap \dots \cap E_{j_t}) = 2r(i+1)+(i+1)-(i+1)^2-1-t = \dim(S_{i+1}^{(i)}(\vr)) - t.
\end{array} 
\end{equation}
Now, consider the variety $\xnb$ as a subvariety of the variety $\mathcal{Q}(2r-1)_{r-1}$ in Construction \ref{ccq}. By Proposition \ref{symplcones} $\xnb$ is smooth away from the exceptional divisors. Furthermore, the exceptional divisor $E_i^{q}$ in Construction \ref{ccq} intersects $\xnb$ in the exceptional divisor $E_i$ in Construction \ref{ccssf}. By Proposition \ref{symplcones} $E_i\rightarrow S_{i}^{(i-1)}(\mathcal{V}_2^{2r-1})$ is a fibration with $\mathcal{S}_{2(r-i)}$ as fiber. Hence $E_i^q\cap \xnb$ is a smooth divisor in $\xnb$ and therefore $\xnb$ is smooth. 

Now, consider an intersection of the form $E_{j_1}\cap\dots \cap E_{j_t}$ and the fibration
$$E_{j_1}\cap\dots \cap E_{j_t}\rightarrow E_{j_1}\cap\dots E_{j_{t-1}}\cap S_{j_t}^{(j_t-1)}(\vr).$$
By Proposition \ref{symplcones} this fibration has fibers isomorphic to $\mathcal{S}_{2(r-j_t)}$. By the previous part of the proof we have that
$$\dim(E_{j_1}\cap\dots E_{j_{t-1}}\cap S_{j_t}^{(j_t-1)}(\vr))+\dim(\mathcal{S}_{2(r-j_t)}) = r^2+r-t = \dim(\xnb)-t$$
and hence the intersection $E_{j_1}\cap\dots \cap E_{j_t}$  is transversal. Note also that considering the fibration
$$S_{r}^{(r-1)}(\vr) \cap E_{j_1} \cap \dots \cap E_{j_t}\rightarrow E_{j_1}\cap \dots \cap E_{j_{t-1}}\cap S_{j_t}^{(j_t-1)}(\vr)$$
and (\ref{eq_dim}) we get that the intersection $S_{r}^{(r-1)}(\vr) \cap E_{j_1} \cap \dots \cap E_{j_t}$ is transversal as well. 

Finally, for the claim about the orbit closures it is enough to recall that the $\spn$-action on $\xnb$ is the restriction of the $SL(2r)$-action on $\mathcal{Q}(2r-1)_{r-1}$ in (\ref{acss}) and to use the statement about the orbit closures in Proposition \ref{Qn_won}.  
\end{proof}

\begin{Proposition}\label{mult}
We have that
$$ \mult_{S_r(\vr)} S_{2r-1}(\vr) = r.$$
Moreover, if $H_{\xn}$ is the hyperplane section of $\xn$, we have that $S_r(\vr) \sim 2H_{\xn}$.
\end{Proposition}
\begin{proof}
We will compute the tangent cone of $S_r(\vr)$ at the point $p_r = (p_{i,j})_{i,j=0,\dots,2r-1}$, where $p_{i,i}=1$ for $i=0, \dots, r-1$ and $p_{i,j}=0$ otherwise.

Consider the change of coordinates $z_{i,i}\mapsto z_{i,i}-z_{0,0}$ ans set $z_{0,0} = 1$. By Remark \ref{equations} the tangent space of $X_{2r}$ at $p_r$ is cut out by a set of linear equations and among these equations we have $\{z_{i,j} = 0\}$ for $i,j = r,\dots, 2r-2$, and $z_{i,i} = z_{i+1,i+1}$ for $i = r,\dots, 2r-2$.

Now, the tangent cone of $\sec_{2r-1}(\mathcal{V}_{2}^{2r-1})$ is cut out by the determinant of the bottom right $r\times r$ submatrix of the matrix $Z$ in (\ref{matrix}). Note that substituting the relations on the $z_{i,j}$ above in this determinant we get $z_{2r-1,2r-1}^r$. 

By Proposition \ref{sec} $\sec_{2r-1}(\mathcal{V}_{2}^{2r-1})$ and $\sec_{r}(\mathcal{V}_{2}^{2r-1})$ cut out on $X_{2r}$ the same divisor set-theoretically. The previous computation yields that $\sec_{2r-1}(\mathcal{V}_{2}^{2r-1})$ cuts out $\sec_{r}(\mathcal{V}_{2}^{2r-1})\cap X_{2r}$ on $X_{2r}$ with multiplicity $r$.

Now, recall that by Remark \ref{sec_ver} $\deg(\sec_{2r-1}(\mathcal{V}_2^{2r-1}))=2r$. Let $D$ be the divisor $\sec_r(\mathcal{V}_2^{2r-1})\cap X_{2r}$. Finally $\sec_{2r-1}(\mathcal{V}_2^{2r-1})\cap X_{2r}\sim 2rH_{\xn}$ yields $D\sim 2H_{\xn}$. 
\end{proof}

\begin{Remark}
In the case $r = 2$ we worked out explicitly the quadratic polynomial cutting out $S_2(\mathcal{V}_2^3)$ in $X_{4}$ and we got that $S_2(\mathcal{V}_2^3) = X_{4}\cap \{z_{0,3}z_{1,2}+z_{1,3}^2-z_{0,1}z_{2,3}-z_{1,1}z_{3,3} = 0\}$.
\end{Remark}

\section{Divisors on $\mathcal{S}_{2r}$}\label{divS2r}
Let $X$ be a normal projective $\mathbb{Q}$-factorial variety over an algebraically closed field of characteristic zero. We denote by $N^1(X)$ the real vector space of $\mathbb{R}$-Cartier divisors modulo numerical equivalence. 
The \emph{nef cone} of $X$ is the closed convex cone $\Nef(X)\subset N^1(X)$ generated by classes of nef divisors. 

The stable base locus $\textbf{B}(D)$ of a $\mathbb{Q}$-divisor $D$ is the set-theoretic intersection of the base loci of the complete linear systems $|sD|$ for all positive integers $s$ such that $sD$ is integral
\stepcounter{thm}
\begin{equation}\label{sbl}
\textbf{B}(D) = \bigcap_{s > 0}B(sD).
\end{equation}
The \emph{movable cone} of $X$ is the convex cone $\Mov(X)\subset N^1(X)$ generated by classes of 
\emph{movable divisors}. These are Cartier divisors whose stable base locus has codimension at least two in $X$.
The \emph{effective cone} of $X$ is the convex cone $\Eff(X)\subset N^1(X)$ generated by classes of 
\emph{effective divisors}. We have inclusions $\Nef(X)\ \subset \ \overline{\Mov(X)}\ \subset \ \overline{\Eff(X)}$. We refer to \cite[Chapter 1]{De01} for a comprehensive treatment of these topics. 

In this section we will study the Picard rank and the cones of effective and nef divisors of the wonderful compactification $\mathcal{S}_{2r}$. We will need the following result.

\begin{Lemma}\label{lso}
Let $SO(2r)$ be the special orthogonal group. Then  
$$SO(2r) \cap Sp(2r) \cong GL(r).$$
In particular, $SO(2)  \cong GL(1) \cong K^{*}$.
\end{Lemma}
\begin{proof}
Consider the bilinear symmetric form given by the matrix 
$J=\begin{pmatrix}
0_{r,r} & I_{r,r}\\
I_{r,r} & 0_{r,r}
\end{pmatrix}$.
Set $N= \begin{pmatrix}
I_{r,r} & \xi I_{r,r}\\
\frac{1}{2} I_{r,r} & -\frac{\xi}{2} I_{r,r}
\end{pmatrix}$, with $\xi^2 = -1$. Note that $N^t J N = I_{2r,2r}$ and $N^t \Omega N = -\xi\Omega$. Therefore, we may prove the statement for the intersection $SO_J(2r) \cap Sp(2r)$, where $SO_J(2r)$ is the group of determinant one matrices which are orthogonal with respect to $J$. 

Let $M = \begin{pmatrix}
A & B\\ 
C & D
\end{pmatrix}\in \gln$ be a general $2r\times 2r$ invertible matrix, where $A,B,C,D$ are $r \times r$ matrices. Now, $M \in SO_J(2r) \cap \spn$ if and only if
$$\begin{cases}
A^t C=0_{r,r};\\
A^t D=I_{r,r};\\
B^t C=0_{r,r};\\
B^t D=0_{r,r};\end{cases} \rm{that\: is\quad} \begin{cases} D=A^{-t};\\
C=0_{r,r};\\
B=0_{r,r};\end{cases}$$
and hence 
$$\son \cap \spn \cong  SO_J(2r) \cap \spn =\Bigr\{\begin{pmatrix}
A & 0_{r,r}\\
0_{r,r} & A^{-t}
\end{pmatrix} \text{ for } A \in GL(r) \Bigl\} \cong GL(r).$$
For the last claim in the case $r = 1$ it is enough to note that $SO(2) \cap Sp(2) = \sod$. In fact every $2 \times 2$ matrix with determinant one is symplectic.
\end{proof}

\begin{Proposition}\label{pic_G/H}
Let $O_{2r}\subset \xn$ be the orbit of the identity. Then $\Pic(O_{2r}) \cong \mathbb{Z}/2\mathbb{Z}$.
\end{Proposition}
\begin{proof}
The group $G =\spn$ is semi-simple and simply connected. If $H\subset G$ is the stabilizer of the identity then  \cite[Theorem 4.5.1.2]{ADHL15} yields that $\Pic(G/H) \cong \mathbb{X}(H)$, where $\mathbb{X}(H)$ is the group of characters of $H$. We have that
$$H=\{M \in \spn , MM^t = \lambda_M I_{2r,2r} , \text{for some } \lambda_M \in K^{*}\}.$$
Then, for a general element $M \in H$ we have
$$\begin{cases}
MM^t= \lambda_M I_{2r,2r};\\
M^t \Omega M = \Omega;
\end{cases} \Rightarrow \lambda_M M^{-1} \Omega M = \Omega \Rightarrow \lambda_M \Omega M = M \Omega.$$
Let $v$ be an eigenvector of $\Omega$ with eigenvalue $\mu$. Then
$$\lambda_M \Omega M v = M \Omega v = M \mu v = \mu M v.$$
Setting $y= M v$ we have $\Omega y = (\lambda_M^{-1} \mu) y$ and so $y$ is an eigenvector of $\Omega$ with eigenvalue $\lambda_M^{-1} \mu$. The characteristic polynomial of $\Omega$ is $P_\Omega(\lambda)=(\lambda-\xi)^r(\lambda+\xi)^r$ where $\xi^2 = -1$. Therefore the only eigenvalues of $\Omega$ are $\xi$ and $-\xi$. So
$$\begin{cases}
\mu = \pm \xi;\\
\lambda_M^{-1} \mu = \pm \xi;
\end{cases} \Rightarrow \lambda_M^{-1}= \pm 1 \Rightarrow \lambda_M= \pm 1 
$$
and there is a morphism of groups
\begin{align*}
\varphi: H &\longrightarrow \mathbb{Z}/2\mathbb{Z}\\
M &\longmapsto \lambda_M
\end{align*}
The morphism $\varphi$ is surjective. Indeed we have $\varphi(I_{2r,2r})=1$, and if $S = \begin{pmatrix}
0_{r,r} & \xi I_{r,r} \\
\xi I_{r,r} & 0_{r,r}
\end{pmatrix}$ then $S^t \Omega S = \Omega$, $SS^t= - I_{2r,2r}$, $S \in H$ and $\varphi(S)=-1$. This yields an exact sequence
\stepcounter{thm}
\begin{equation}\label{exseq}
1 \rightarrow \overline{H} \rightarrow H \rightarrow \mathbb{Z}/2\mathbb{Z} \rightarrow 1
\end{equation}
where $\overline{H}= \{M \in \spn , MM^t = I_{2r,2r}\}$, and we can write $H= \overline{H} \cup S\overline{H}$. 

As in Lemma \ref{lso}, we consider the bilinear form $J=\begin{pmatrix}
0_{r,r} & I_{r,r}\\
I_{r,r} & 0_{r,r}
\end{pmatrix}$, which is congruent to the bilinear form $I_{2r,2r}$ via the matrix $N= \begin{pmatrix}
I_{r,r} & \xi I_{r,r}\\
\frac{1}{2} I_{r,r} & -\frac{\xi}{2} I_{r,r}
\end{pmatrix}$, where $\xi^2=-1$. Set $\overline{H}_J= \{M \in \spn , M J M^t = J\}$ and $H_J=\{M \in \spn , M J M^t = \lambda_M J , \text{for some } \lambda_M \in K^{*}\}$. There is an isomorphism  
\begin{align*}
\alpha: H &\longrightarrow H_J \\
M &\longmapsto N M N^{-1}
\end{align*}
such that $\alpha(\overline{H}) = \overline{H}_J$, $\tilde{S}:=\alpha(S)= \begin{pmatrix}
0 & -2 I_{r,r} \\
\frac{1}{2} I_{r,r} & 0
\end{pmatrix}$ and $H_J = \overline{H}_J \cup \tilde{S} \overline{H}_J$.
Take $B\in H_J$ and consider $\alpha^{-1}(B)\in H$. By the first part of the proof there is a morphism of groups $H_J\rightarrow\mathbb{Z}/2\mathbb{Z}$ mapping $B$ to $\lambda_{\alpha^{-1}(B)}$, and fitting in the following exact sequence 
$$
1 \rightarrow \overline{H}_J \rightarrow H_J \rightarrow \mathbb{Z}/2\mathbb{Z} \rightarrow 1
$$
Since $H_J/\overline{H}_J$ is abelian the commutator $[H_J,H_J]$ of $H_J$ is contained in $\overline{H}_J$. By the proof of Lemma \ref{lso} we have that an element $h \in \overline{H}_J$ is of the form $h=\begin{pmatrix}
A & 0_{r,r}\\
0_{r,r} & A^{-t}
\end{pmatrix} \text{ for } A \in GL(r)$. Then $h^{-1}=\begin{pmatrix}
A^{-1} & 0_{r,r}\\
0_{r,r} & A^{t}
\end{pmatrix} \text{ for } A \in GL(r)$. Furthermore $\tilde{S}^{-1} = \begin{pmatrix}
0_{r,r} & 2 I_{r,r} \\
-\frac{1}{2} I_{r,r} & 0_{r,r}
\end{pmatrix}$. Therefore
\begin{align*}
[\tilde{S},h] &= \tilde{S}h\tilde{S}^{-1}h^{-1} = \begin{pmatrix}
0_{r,r} & -2 I_{r,r} \\
\frac{1}{2} I_{r,r} & 0_{r,r}
\end{pmatrix} \begin{pmatrix}
A & 0_{r,r}\\
0_{r,r} & A^{-t}
\end{pmatrix} \begin{pmatrix}
0_{r,r} & 2 I_{r,r} \\
-\frac{1}{2} I_{r,r} & 0_{r,r}
\end{pmatrix} \begin{pmatrix}
A^{-1} & 0_{r,r}\\
0_{r,r} & A^{t}
\end{pmatrix} =\begin{pmatrix}
A^{-t}A^{-1} & 0_{r,r}\\
0_{r,r} & A A^{t}
\end{pmatrix}.
\end{align*}
Setting $B = A^{-t}A^{-1}$, we have $B^{-t}=(A^{-t}A^{-1})^{-t}=AA^t$ with $B \in \glnm$ symmetric. So $[H_J,H_J]$ is the subgroup of $\overline{H}_J \cong GL(r)$ generated by symmetric matrices and since by \cite[Theorem 1]{Bos86} all $r \times r$ matrices can be written as product of symmetric matrices we get $[H_J,H_J]=\overline{H}_J$. 

Then, $H / [H,H] \cong H_J / [H_J,H_J] \cong H_J / \overline{H}_J \cong H/\overline{H}$ and by the exact sequence (\ref{exseq}) we have $H/\overline{H} \cong \mathbb{Z}/2\mathbb{Z}$. Finally, by \cite[Lemma 22.2]{Bur65} $\mathbb{X}(H) \cong \mathbb{X}(H/[H,H])$, and hence $\Pic(G/H) \cong \mathbb{X}(H) \cong \mathbb{Z}/2 \mathbb{Z}$.
\end{proof}

Now, we are ready to compute the Picard rank and the colors of the wonderful variety $\mathcal{S}_{2r}$.

\begin{Proposition}\label{pic_x2r}
The Picard rank of $\mathcal{S}_{2r}$ is $\rho(\mathcal{S}_{2r}) = r$.
\end{Proposition}
\begin{proof}
As before set $G=\spn$ and let $H$ be the stabilizer of the identity. By Theorem \ref{main1} the variety $\mathcal{S}_{2r}$ is wonderful with boundary divisors $E_1,\dots, E_{r-1},S_r^{(r-1)}(\mathcal{V}_2^{2r-1})$. By \cite[Proposition 2.2.1]{Br07} there is an exact sequence
$$ 0 \rightarrow \mathbb{Z}^r \rightarrow \Pic(\mathcal{S}_{2r}) \rightarrow \Pic(G/H) \rightarrow 0$$
Hence, Proposition \ref{pic_G/H} yields that the Picard rank of $\mathcal{S}_{2r}$ is $r$.
\end{proof}

For $i=1,\dots,r$ we define the divisors $D_i$ as the strict transforms in $\mathcal{S}_{2r}$ of the divisor given by the intersection of  
$$\det \begin{pmatrix}
z_{0,0} & \dots & z_{0,i-1}\\
\vdots & \ddots & \vdots \\
z_{0,i-1} & \dots & z_{i-1,i-1}\\
\end{pmatrix}=0$$
with $X_{2r}$. 

\begin{Proposition}\label{bound_col}
The set of boundary divisors of $\mathcal{S}_{2r}$ is $\{E_1,\dots,E_{r-1},S_r^{r-1}(\mathcal{V}_2^{2r-1})\}$ while the set of colors of $\mathcal{S}_{2r}$ is $\{D_1,\dots,D_r\}$. 
\end{Proposition}
\begin{proof}
The claim on the set of boundary divisors follows from Theorem \ref{main1}. We compute the colors. We first prove that $D_r \subseteq \xnb$ is stabilized by the Borel subgroup. Consider a matrix $Z=\begin{pmatrix}
Z_{0,0} & Z_{0,1} \\
Z_{0,1} & Z_{1,1}
\end{pmatrix}$ where the $Z_{i,j}$ are $r \times r$ matrices. Let $M =\begin{pmatrix}
A & 0_{r,r}\\
B & A^{-t}
\end{pmatrix} \in \mathscr{B}$, then
\begin{align*}
\bar{Z} = M \cdot Z \cdot M^t & =
 \begin{pmatrix}
AZ_{0,0}A^t &  A Z_{0,0} B^t+ A Z_{0,1} A^{-1}\\  
B Z_{0,0} A^t + A^{-t} Z_{0,1} A^t & B Z_{0,0} B^t + A^{-t} Z_{0,1} B^t  + B Z_{0,1} A^{-1} + A^{-t}Z_{1,1} A^{-1} 
\end{pmatrix} 
\end{align*}
and $\det(\bar{Z}_{0,0})=\det(AZ_{0,0}A^t)=\det(A)^2 \det(Z_{0,0})$ where $\det(A) \neq 0$ since $A \in \glnm$. Therefore, $D_r$ is stabilized by the Borel subgroup.

We focus now on the block $\bar{Z}_{0,0}$ of the matrix $\bar{Z}$. We divide the matrices $A$ and $Z_{0,0}$ respectively in blocks $A_{j,k}$, $W_{j,k}$ of matrices $j \times k$ as follows
$A= \begin{pmatrix}
A_{i,i}& A_{i,r-i} \\
A_{r-i,i} & A_{r-i,r-i}
\end{pmatrix}$ and $Z_{0,0}=\begin{pmatrix}
W_{i,i} & W_{i,r-i} \\
W_{r-i,i} & W_{r-i,r-i}
\end{pmatrix}$. Recall that by Remark \ref{borelsym} the matrix $A$ is lower triangular. We have
$
\bar{Z}_{0,0} 
= \begin{pmatrix}
\bar{W}_{i,i} & \bar{W}_{i,r-i} \\
\bar{W}_{r-i,i} & \bar{W}_{r-i,r-i}
\end{pmatrix}
$
with $\bar{W}_{i,i} = A_{i,i}W_{i,i}A_{i,i}^t$. The divisor $D_i$ is defined by $\det(W_{i,i})=0$ and since $\det(A)= \det(A_{i,i})\det(A_{r-i,r-i})\neq 0$ we get that $D_i$ is stabilized by $\mathscr{B}$ for $i=1, \dots r$.

As noticed in \cite[Remark 4.5.5.3]{ADHL15}, if $(X,\mathscr{G},\mathscr{B},x_0)$ is a spherical wonderful variety with colors $D_1,\dots,D_s$ the big cell $X\setminus (D_1\cup\dots \cup D_s)$ is an affine space. Therefore, it admits only constant invertible global functions and $\Pic(X)$ is generated by $D_1,\dots,D_s$. 

Therefore, in order to conclude that we found all the colors of $\mathcal{S}_{2r}$ it is enough to recall that by Proposition \ref{pic_x2r} $\mathcal{S}_{2r}$ has Picard rank $r$.
\end{proof}

In the following we will denote by $H$ the pull-back in $\xnb$ of the hyperplane section of $\xn$. By Proposition \ref{pic_x2r} $H,E_1\dots,E_{r-1}$ generate $\Pic(\xnb)$.

\begin{Proposition}\label{eff_nef}
The extremal rays of $\Eff(\xnb)$ are generated by $E_1,\dots,E_{r-1},S_r^{r-1}(\mathcal{V}_2^{2r-1})$ and the extremal rays of $\Nef(\xnb)$ are generated by $D_1,\dots,D_r$.
\end{Proposition}
\begin{proof}
By \cite[Proposition 4.5.4.4]{ADHL15} and Proposition \ref{bound_col} $\Eff(\xnb)$ is generated by $E_1,\dots,E_{r-1},S_r^{r-1}(\mathcal{V}_2^{2r-1})$ and $D_1,\dots,D_r$. 

Note that by Constructions \ref{ccq} and \ref{ccssf} there as an inclusion $i:\xnb\rightarrow\mathcal{Q}(2r-1)_{r-1}$ inducing an isomorphism of the Picard groups. By \cite[Section 2]{Ce15} the linear system on $\mathcal{Q}(2r-1)_{r-1}$ that restricts to the linear system of $D_i$ on $\xnb$ induces a birational morphism $\mathcal{Q}(2r-1)_{r-1}\rightarrow W_i$ whose exceptional locus is contained in the union of the exceptional divisors in Construction \ref{ccq}. Therefore, $D_i$ induces a birational morphism $\xnb\rightarrow Z_i$ and hence $D_i$ lies in the interior of the effective cone of $\xnb$ for any $i = 1,\dots, r$. This proves that the effective cone of $\xnb$ is generated by $E_1,\dots,E_{r-1},S_r^{r-1}(\mathcal{V}_2^{2r-1})$. Finally, by \cite[Section 2.6]{Br89} $D_1,\dots,D_r$ generate the extremal rays of the nef cone.
\end{proof}

In order to study the birational geometry of $\mathcal{S}_{2r}$ we will need the following result.

\begin{Proposition}\label{cones_col}
Let $H_i^r$ be the divisor in $X_{2r}\subset\mathbb{P}^N$ cut out by the determinant of the $i\times i$ top left submatrix of the matrix $Z$ in (\ref{matrix}). The tangent cone of $H_i^r$ at a point of $S_k(\vr)\setminus S_{k-1}(\vr)$ for $i=2, \dots,r$ and $k <i$ is a cone with vertex of dimension $k(2r+1-k)$ over $H_{i-k}^{r-k}$. 
\end{Proposition}
\begin{proof}
It is enough to note that the tangent cone of $H_i^r$ at the point $p_k = (p_{i,j})_{i,j=0,\dots,2r-1}$, where $p_{i,i}=1$ for $i=0, \dots, k-1$ and $p_{i,j}=0$ otherwise, is cut out by 
$$\det \left(\begin{array}{cccc}
z_{k,k} & z_{k,k+1} & \hdots & z_{k,i-1} \\ 
z_{k,k+1} & z_{k+1,k+1} & \hdots & z_{k+1,i-1} \\ 
\vdots & \vdots & \ddots & \vdots \\ 
z_{k,i-1} & z_{k+1,i-1} & \hdots & z_{i-1,i-1}
\end{array}\right) = 0$$
and by the equations for the tangent cone of $X_{2r}$ in the proof of Proposition \ref{symplcones}.
\end{proof}

\section{Birational geometry of $\mathcal{S}_{2r}$}\label{birS2r}
The stable base locus of an effective $\mathbb{Q}$-divisor on a normal $\mathbb{Q}$-factorial projective variety $X$ has been defined in (\ref{sbl}). Since stable base loci do not behave well with respect to numerical equivalence \cite[Example 10.3.3]{La04II}, we will assume that $h^{1}(X,\mathcal{O}_X)=0$ so that linear and numerical equivalence of $\mathbb{Q}$-divisors coincide. 

Then numerically equivalent $\mathbb{Q}$-divisors on $X$ have the same stable base locus, and the pseudo-effective cone $\overline{\Eff}(X)$ of $X$ can be decomposed into chambers depending on the stable base locus of the corresponding linear series. The resulting decomposition is called \textit{stable base locus decomposition}. 

\begin{Remark}\label{SBLMC}
Recall that two divisors $D_1,D_2$ are said to be \textit{Mori equivalent} if $\textbf{B}(D_1) = \textbf{B}(D_2)$ and the following diagram of rational maps is commutative
   \[
  \begin{tikzpicture}[xscale=1.5,yscale=-1.2]
    \node (A0_1) at (1, 0) {$X$};
    \node (A1_0) at (0, 1) {$X(D_1)$};
    \node (A1_2) at (2, 1) {$X(D_2)$};
    \path (A1_0) edge [->]node [auto] {$\scriptstyle{}$} node [rotate=180,sloped] {$\scriptstyle{\widetilde{\ \ \ }}$} (A1_2);
    \path (A0_1) edge [->,dashed]node [auto] {$\scriptstyle{\phi_{D_2}}$} (A1_2);
    \path (A0_1) edge [->,swap, dashed]node [auto] {$\scriptstyle{\phi_{D_1}}$} (A1_0);
  \end{tikzpicture}
  \]
where the horizontal arrow is an isomorphism. Therefore, the Mori chamber decomposition is a, possibly trivial, refinement of the stable base locus decomposition.
\end{Remark}

Let $X$ be a normal $\mathbb{Q}$-factorial variety with free and finitely generated divisor class group $\Cl(X)$. Fix a subgroup $G$ of the group of Weil divisors on $X$ such that the canonical map $G\rightarrow\Cl(X)$, mapping a divisor $D\in G$ to its class $[D]$, is an isomorphism. The \textit{Cox ring} of $X$ is defined as
$$\Cox(X) = \bigoplus_{[D]\in \Cl(X)}H^0(X,\mathcal{O}_X(D))$$
where $D\in G$ represents $[D]\in\Cl(X)$, and the multiplication in $\Cox(X)$ is defined by the standard multiplication of homogeneous sections in the field of rational functions on $X$. If $\Cox(X)$ is finitely generated as an algebra over the base field, then $X$ is said to be a \textit{Mori dream space}. A perhaps more enlightening definition, especially for the relation with the minimal model program, is the following. 

\begin{Definition}\label{def:MDS} 
A normal projective $\mathbb{Q}$-factorial variety $X$ is called a \emph{Mori dream space}
if the following conditions hold:
\begin{enumerate}
\item[-] $\Pic{(X)}$ is finitely generated, or equivalently $h^1(X,\mathcal{O}_X)=0$,
\item[-] $\Nef{(X)}$ is generated by the classes of finitely many semi-ample divisors,
\item[-] there is a finite collection of small $\mathbb{Q}$-factorial modifications
 $f_i: X \dasharrow X_i$, such that each $X_i$ satisfies the second condition above, and $
 \Mov{(X)} \ = \ \bigcup_i \  f_i^*(\Nef{(X_i)})$.
\end{enumerate}
\end{Definition}

The collection of all faces of all cones $f_i^*(\Nef{(X_i)})$ above forms a fan which is supported on $\Mov(X)$.
If two maximal cones of this fan, say $f_i^*(\Nef{(X_i)})$ and $f_j^*(\Nef{(X_j)})$, meet along a facet,
then there exist a normal projective variety $Y$, a small modification $\varphi:X_i\dasharrow X_j$, and $h_i:X_i\rightarrow Y$, $h_j:X_j\rightarrow Y$ small birational morphisms of relative Picard number one such that $h_j\circ\varphi = h_i$. The fan structure on $\Mov(X)$ can be extended to a fan supported on $\Eff(X)$ as follows. 

\begin{Definition}\label{MCD}
Let $X$ be a Mori dream space.
We describe a fan structure on the effective cone $\Eff(X)$, called the \emph{Mori chamber decomposition}.
We refer to \cite[Proposition 1.11]{HK00} and \cite[Section 2.2]{Ok16} for details.
There are finitely many birational contractions from $X$ to Mori dream spaces, denoted by $g_i:X\dasharrow Y_i$.
The set $\Exc(g_i)$ of exceptional prime divisors of $g_i$ has cardinality $\rho(X/Y_i)=\rho(X)-\rho(Y_i)$.
The maximal cones $\mathcal{C}$ of the Mori chamber decomposition of $\Eff(X)$ are of the form: $\mathcal{C}_i \ = \left\langle g_i^*\big(\Nef(Y_i)\big) , \Exc(g_i) \right\rangle$. We call $\mathcal{C}_i$ or its interior $\mathcal{C}_i^{^\circ}$ a \emph{maximal chamber} of $\Eff(X)$.
\end{Definition}

If $X$ is a Mori dream space, satisfying then the condition $h^1(X,\mathcal{O}_X)=0$, determining the stable base locus decomposition of $\Eff(X)$ is a first step in order to compute its Mori chamber decomposition. 

\begin{Remark}\label{sphMDS}
By the work of M. Brion \cite{Br93} we have that $\mathbb{Q}$-factorial spherical varieties are Mori dream spaces. An alternative proof of this result can be found in \cite[Section 4]{Pe14}. In particular, by Theorem \ref{main1} the wonderful compactification $\mathcal{S}_{2r}$ is a Mori dream space.
\end{Remark}

\begin{Remark}\label{toric}
Recall that by \cite[Proposition 2.11]{HK00} given a Mori Dream Space $X$ there is an embedding $i:X\rightarrow \mathcal{T}_X$ into a simplicial projective toric variety $\mathcal{T}_X$ such that $i^{*}:\Pic(\mathcal{T}_X)\rightarrow \Pic(X)$ is an isomorphism inducing an isomorphism $\Eff(\mathcal{T}_X)\rightarrow \Eff(X)$. Furthermore, the Mori chamber decomposition of $\Eff(\mathcal{T}_X)$ is a refinement of the Mori chamber decomposition of $\Eff(X)$. Indeed, if $\Cox(X) \cong \frac{K[T_1,\dots,T_s]}{I}$ where the $T_i$ are homogeneous generators with non-trivial effective $\Pic(X)$-degrees then $\Cox(\mathcal{T}_X)\cong K[T_1,\dots,T_s]$.

Since the variety $\mathcal{T}_{X}$ is toric, the Mori chamber decomposition of $\Eff(\mathcal{T}_{X})$ can be computed by means of the Gelfand–Kapranov–Zelevinsky, GKZ for short, decomposition \cite[Section 2.2.2]{ADHL15}. Let us consider the family $\mathcal{W}$ of vectors in $\Pic(\mathcal{T}_{X})$ given by the generators of $\Cox(\mathcal{T}_{X})$, and let $\Omega(\mathcal{W})$ be the set of all convex polyhedral cones generated by some of the vectors in $\mathcal{W}$. By \cite[Construction 2.2.2.1]{ADHL15} the GKZ chambers of $\Eff(\mathcal{T}_{X})$ are given by the intersections of all the cones in $\Omega(\mathcal{W})$ containing a fixed divisor in $\Eff(\mathcal{T}_{X})$.
\end{Remark}

\begin{Remark}\label{gen_cox}
Let $(X,\mathscr{G},\mathscr{B},x_0)$ be a projective spherical variety. Consider a divisor $D$ on $X$, and let $f_D$ be the, unique up to constants, section of $\mathcal{O}_X(D)$ associated to $D$. We will denote by $\lin_K(\mathscr{G}\cdot D)\subseteq\Cox(X)$ the finite-dimensional vector subspace of $\Cox(X)$ spanned by the orbit of $f_D$ under the action of $\mathscr{G}$ that is the smallest linear subspace of $\Cox(X)$ containing the $\mathscr{G}$-orbit of $f_D$.

By \cite[Theorem 4.5.4.6]{ADHL15} if $\mathscr{G}$ is a semi-simple and simply connected algebraic group and $(X,\mathscr{G},\mathscr{B},x_0)$ is a spherical variety with boundary divisors $E_1,\dots,E_r$ and colors $D_1,\dots,D_s$ then $\Cox(X)$ is generated as a $K$-algebra by the canonical sections of the $E_i$'s and the finite dimensional vector subspaces $\lin_{K}(\mathscr{G}\cdot D_i)\subseteq \Cox(X)$ for $1\leq i\leq s$.
\end{Remark}
\begin{Definition}
Let $X$ be a normal projective $\mathbb{Q}$-factorial variety. We say that $X$ is weak Fano if $-K_X$ is nef and big. 
\end{Definition}

By \cite[Corollary 1.3.2]{BCHM} a weak Fano variety is a Mori dream space.

\begin{Remark}\label{stpb}
Let $Y$ be a smooth and irreducible subvariety of a smooth variety $X$, and let $f:Bl_YX\rightarrow X$ be the blow-up of $X$ along $Y$ with exceptional divisor $E$. Then for any divisor $D\in \Pic(X)$ in $\Pic(Bl_YX)$ we have
$$\widetilde{D} \sim f^{*}D-\mult_{Y}(D) E$$
where $\widetilde{D}\subset Bl_YX$ is the strict transform of $D$, and $\mult_Y(D)$ is the multiplicity of $D$ at a general point of $Y$. 
\end{Remark}

\begin{Corollary}\label{Cox_gen}
The Cox ring of $\mathcal{S}_{2r}$ is generated by the sections of $D_1,\dots D_r,E_1,\dots, E_{r-1}, S^{(r-1)}_r(\mathcal{V}_r^{2r-1})$.
\end{Corollary}
\begin{proof}
This follows from Proposition \ref{bound_col} and Remark \ref{gen_cox}.
\end{proof}

Our aim is to study the Mori chamber decomposition of the wonderful compactification $\mathcal{S}_{2r}$. Since $\mathcal{S}_{2}\cong\mathbb{P}^2$ the first interesting case is for $r = 2$.

\begin{Proposition}\label{mcd_4}
For the variety $\mathcal{S}_4$ we have that $\Pic(\mathcal{S}_4)$ is generated by $D_1, E_1$. Furthermore, $D_1\sim H$, $D_2 \sim 2H-E_1$, $S_2^{(1)}(\mathcal{V}_2^3) \sim 2 H - 2 E_1$, and $\Cox(\mathcal{S}_4)$ is generated by the sections of $D_1,D_2,E_1,S^{(1)}_2(\mathcal{V}_2^3)$. The Mori chamber decomposition of $\Eff(\mathcal{S}_4)$ has three chambers as displayed in the following picture:
$$
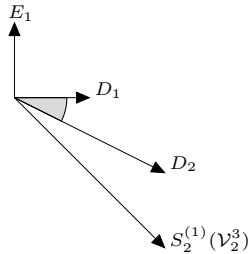
\begin{tikzpicture}[line cap=round,line join=round,>=triangle 45,x=1.0cm,y=1.0cm]
xmin=2.5,
xmax=7.5,
ymin=1.75,
ymax=5.301934941656083,
xtick={2.5,3.0,...,5.5},
ytick={2.0,2.5,...,5.0},]
\clip(2.1,1.65) rectangle (7.5,5.301934941656083);
\draw [->,line width=0.1pt] (3.,4.) -- (3.,5.);
\draw [->,line width=0.1pt] (3.,4.) -- (4.,4.);
\draw [->,line width=0.1pt] (3.,4.) -- (5.,3.);
\draw [->,line width=0.1pt] (3.,4.) -- (5.,2.);
\draw [shift={(3.,4.)},line width=0.4pt,fill=black,fill opacity=0.15000000596046448]  (0,0) --  plot[domain=-0.46364760900080615:0.,variable=\t]({1.*0.6965067581669063*cos(\t r)+0.*0.6965067581669063*sin(\t r)},{0.*0.6965067581669063*cos(\t r)+1.*0.6965067581669063*sin(\t r)}) -- cycle ;
\begin{scriptsize}
\draw[color=black] (3.085808915158281,5.121548796059524) node {$E_1$};
\draw[color=black] (4.25,4.1) node {$D_1$};
\draw[color=black] (5.25,3.119262579937719) node {$D_2$};
\draw[color=black] (5.6,2.118119471876817) node {$S_2^{(1)}(\mathcal{V}_2^3)$};
\end{scriptsize}
\end{tikzpicture}
$$
and the movable cone coincides with the nef cone generated by $D_1$ and $D_2$.
\end{Proposition}
\begin{proof}
Since $\mathcal{S}_4$ is the blow-up of a smooth variety along a smooth subvariety the relations $D_2 \sim 2H-E_1$, $S_2^{(1)}(\mathcal{V}_2^3) \sim 2 H - 2 E_1$ follow from Propositions \ref{symplcones}, \ref{mult}, \ref{cones_col} and Remark \ref{stpb}.

The statement on the generators of the Cox ring follows from Corollary \ref{Cox_gen}. Furthermore, by Remarks \ref{toric} and \ref{gen_cox} the Mori chamber decomposition of $\Eff(\mathcal{S}_4)$ is a, possibly trivial, coarsening of the decomposition in the statement. On the other hand, by Proposition \ref{eff_nef} we know that $H$ and $2H-E_1$ generate $\Nef(\mathcal{S}_4)$ while $E_1$ and $2H-2E_1$ generate $\Eff(\mathcal{S}_4)$. So no ray can be removed and the above decomposition coincides with the Mori chamber decomposition of $\Eff(\mathcal{S}_4)$.
\end{proof}

Next, we consider the case $r=3$.

\begin{Lemma}\label{rel_S6}
For the variety $\mathcal{S}_6$ the Picard group $\Pic(\mathcal{S}_6)$ is generated by $H, E_1, E_2$, and we have the following relations: $D_1\sim H$, $D_2 \sim 2H-E_1$, $D_3 \sim 3H -2E_1-E_2$ and $S_3^2(\mathcal{V}_2^5) \sim 2H-2E_1-2E_2$.
\end{Lemma}
\begin{proof}
Recall, that the first blow-up $f_1:X_6^{(1)}\rightarrow X_6$ in Construction \ref{ccssf} is the blow-up of $X_6$ along the Veronese variety $\mathcal{V}_2^5$ which by Proposition \ref{symplcones} is the singular locus of $X_6$. Hence, in this case we can not use Remark \ref{stpb} to compute the discrepancies of the relevant divisors with respect to $E_1$. In order to do this we consider the line $L = \{z_{1,1}-z_{0,1}=z_{1,1}-z_{2,2}=z_{0,2}=z_{0,3}=z_{0,4}=z_{0,5}=z_{1,2}=z_{1,3}=z_{1,4}=z_{1,5}=z_{2,3}=z_{2,4}=z_{2,5}=z_{3,3}=z_{3,4}=z_{3,5}=z_{4,4}=z_{4,5}=z_{5,5} = 0\}$ and let $\widetilde{L}$ be its strict transform in $X_6^{(1)}$. Slightly abusing the notation we will denote by $D_i$ also the strict transform in $X_6^{(1)}$ of the divisor $H_i^3$ in Proposition \ref{cones_col} for $i = 1,2,3$ and by $H$ the pull-back of the hyperplane section to $X_6^{(1)}$. Clearly, $D_1\sim H$.

Now, let us write $D_2\sim 2H-aE_1$. Note that the line $L$ intersects $\mathcal{V}_2^5$ just at the point $p = [1:0\dots:0]$, and by Remark \ref{equations} and Proposition \ref{symplcones} $L\subset X_6$. By Proposition \ref{symplcones} the tangent cone of $X_6$ at $p$ is a cone over $X_4\cong\mathbb{G}(1,4)$ with $5$-dimensional vertex and $\widetilde{L}$ intersects $E_1$ just at the point $q = [1:0:0:0:1:0:\dots :0]$ of $X_4$. Hence $\widetilde{L}\cdot E_1 = 1$. The divisor $H_2^3$ intersects $L$ in $p$ and in another point not lying on $\mathcal{V}_2^5$. Moreover, by Proposition \ref{cones_col} the tangent cone of $H_2^3$ at $p$ is a hyperplane section of $X_4$ not passing through $q$. Then $\widetilde{L}\cdot D_2 = 1$. By the projection formula we have
$$1 = \widetilde{L}\cdot D_2 = 2\widetilde{L}\cdot H-a\widetilde{L}\cdot E_1 = 2L\cdot H_1^3-a = 2-a$$
and hence $a = 1$. So we may write $D_2\sim 2H-E_1$.

Now, write $D_3\sim 3H-bE_1$. The divisor $H_3^3$ intersects $L$ in $p$ with multiplicity two and in another point not lying on $\mathcal{V}_2^5$. By Proposition \ref{cones_col} the tangent cone of $H_3^3$ at $p$ is a quadratic section of $X_4$ not passing through $q$. Hence
$$1 = \widetilde{L}\cdot D_3 = 3\widetilde{L}\cdot H-a\widetilde{L}\cdot E_1 = 3L\cdot H_1^3-a = 3-a$$
and $a = 2$. Then $D_3\sim 3H-2E_1$.

We will denote by $S_3$ the strict transform of $S_3(\mathcal{V}_2^5)$ in $X_6^{(1)}$. Let $R\subset X_4\cong\mathbb{G}(1,4)$ be a general line. Note that $R$ is contracted by the blow-down morphism and hence 
$$1 = R\cdot D_2 = 2R\cdot H-R\cdot E_1 = -R\cdot E_1$$
yields $R\cdot E_1 = -1$. By Proposition \ref{mult} we may write $S_3\sim 2H-cE_1$ and since by Proposition \ref{symplcones} the tangent cone of $S_3(\mathcal{V}_2^5)$ at a point of $\mathcal{V}_2^5$ is a quadratic section of $X_4$ we have $R\cdot S_3 = 2$. This yields 
$$2 = R\cdot S_3 = 2R\cdot H-cR\cdot E_1 = -cR\cdot E_1 = c$$
and $S_3\sim 2H-2E_1$. 

Now, by Proposition \ref{symplcones} the morphism $f_2:\mathcal{S}_6\rightarrow X_6^{(1)}$ in Construction \ref{ccssf} is the blow-up of a smooth variety along a smooth subvariety. So we can apply Remark \ref{stpb} in order to compute the discrepancies of the divisors with respect to $E_2$. Finally, again by Proposition \ref{symplcones} we get the claim. 
\end{proof}

\begin{thm}\label{dec_S6}
The Cox ring of $\mathcal{S}_6$ is generated by the sections of $D_1,D_2,D_3,E_1,E_2,S_3^{(2)}(\mathcal{V}_2^5)$. The Mori chamber decomposition of the effective cone of $\mathcal{S}_6$ has nine chambers as displayed in the following $2$-dimensional section of $\Eff(\mathcal{S}_6)$:
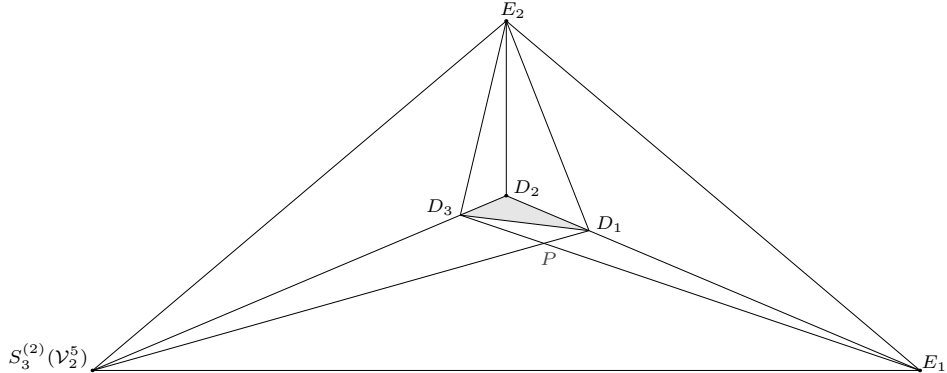
$$
\definecolor{uuuuuu}{rgb}{0.26666666666666666,0.26666666666666666,0.26666666666666666}
\begin{tikzpicture}[xscale=5.5,yscale=4.65][line cap=round,line join=round,>=triangle 45,x=5.0cm,y=5.0cm]
\clip(-1.4,-0.03) rectangle (1.2,1.06);
\fill[line width=0.4pt,fill=black,fill opacity=0.10000000149011612] (-0.1111111111111111,0.4444444444444444) -- (0.,0.5) -- (0.2,0.4) -- cycle;
\draw [line width=0.1pt] (-1.,0.)-- (0.,1.);
\draw [line width=0.1pt] (1.,0.)-- (0.,1.);
\draw [line width=0.1pt] (-1.,0.)-- (1.,0.);
\draw [line width=0.1pt] (-0.1111111111111111,0.4444444444444444)-- (0.,0.5);
\draw [line width=0.1pt] (0.,0.5)-- (0.2,0.4);
\draw [line width=0.1pt] (0.2,0.4)-- (-0.1111111111111111,0.4444444444444444);
\draw [line width=0.1pt] (0.,0.5)-- (0.,1.);
\draw [line width=0.1pt] (0.2,0.4)-- (1.,0.);
\draw [line width=0.1pt] (-0.1111111111111111,0.4444444444444444)-- (-1.,0.);
\draw [line width=0.1pt] (-0.1111111111111111,0.4444444444444444)-- (1.,0.);
\draw [line width=0.1pt] (0.2,0.4)-- (-1.,0.);
\draw [line width=0.1pt] (-0.1111111111111111,0.4444444444444444)-- (0.,1.);
\draw [line width=0.1pt] (0.,1.)-- (0.2,0.4);
\begin{scriptsize}
\draw [fill=black] (-1.,0.) circle (0.1pt);
\draw[color=black] (-1.107,0.04) node {$S_3^{(2)}(\mathcal{V}_2^5)$};
\draw [fill=black] (0.,1.) circle (0.1pt);
\draw[color=black] (0.016094824612243562,1.03) node {$E_2$};
\draw [fill=black] (1.,0.) circle (0.1pt);
\draw[color=black] (1.035,0.021295712735575137) node {$E_1$};
\draw [fill=black] (-0.1111111111111111,0.4444444444444444) circle (0.0pt);
\draw[color=black] (-0.16,0.46729156695787455) node {$D_3$};
\draw [fill=black] (0.,0.5) circle (0.1pt);
\draw[color=black] (0.05,0.5217643430460943) node {$D_2$};
\draw [fill=black] (0.2,0.4) circle (0.0pt);
\draw[color=black] (0.25,0.419) node {$D_1$};
\draw [fill=uuuuuu] (0.09090909090909086,0.36363636363636365) circle (0.0pt);
\draw[color=uuuuuu] (0.102,0.319) node {$P$};
\end{scriptsize}
\end{tikzpicture}
$$
where $P\sim 3H-E_1-E_2$ and $\Mov(\mathcal{S}_6)$ is generated by $D_1,D_2,D_2$ and $P$.
\end{thm}
\begin{proof}
The computation of the movable cone follows from \cite[Proposition 3.3.2.3]{ADHL15}, Proposition \ref{bound_col} and Remark \ref{gen_cox}, and the statement on the generators of $\Cox(\mathcal{S}_6)$ follows from Corollary \ref{Cox_gen}.

Furthermore, by Lemma \ref{rel_S6}, Proposition \ref{bound_col} and Remarks \ref{toric}, \ref{gen_cox} the Mori chamber decomposition of $\Eff(\mathcal{S}_6)$ is a, possibly trivial, coarsening of the decomposition in the statement. 

Note that the stable base loci of a divisor in the interior of chamber delimited by $S_3^{(2)}(\mathcal{V}_2^5),P,E_1$; $S_3^2(\mathcal{V}_2^5),P,D_3$; $S_3^{(2)}(\mathcal{V}_2^5),D_3,E_2$; $D_2,D_3,D_1,E_2$; $E_1,D_1,E_2$,; $P,D_1,E_1$ are respectively given by $S_3^{(2)}(\mathcal{V}_2^5)\cup E_1$; $S_3^2(\mathcal{V}_2^5)$; $E_2\cup S_3^{(2)}(\mathcal{V}_2^5)$; $E_2$; $E_1\cup E_2$; $E_1$. Furthermore, since Mori chambers are convex the stable base locus chamber delimited by $D_2,D_3,D_1,E_2$ must be divided in two Mori chambers by the wall joining $D_2$ and $E_2$. Hence the decomposition in the statement gives the Mori chamber decomposition of $\Eff(\mathcal{S}_6)$ outside of the movable cone. 

Finally, note that the only modifications we could perform inside the movable cone are removing the wall joining $D_1$ and $D_3$ and adding a wall joining $D_2$ and $P$. However, both these modifications are not allowed since by Proposition \ref{eff_nef} the chamber delimited by $D_1,D_2,D_3$ is the nef cone of $\mathcal{S}_6$. 
\end{proof}

\section{Moduli spaces of conics in Lagrangian Grassmannians}\label{K_Grass}
An $n$-pointed rational pre-stable curve $(C,(x_{1},...,x_{n}))$ is a projective, connected, reduced rational curve with at most nodal singularities of arithmetic genus zero, with $n$ distinct and smooth marked points $x_1,...,x_n\in C$. We will refer to the marked and the singular points of $C$ as special points.

Let $X$ be a homogeneous variety. A map $(C,(x_{1},...,x_{n}),\alpha)$, where $\alpha:C\rightarrow X$ is a morphism from an $n$-pointed rational pre-stable curve to $X$, is stable if any component $E\cong\mathbb{P}^{1}$ of $C$ contracted by $\alpha$ contains at least three special points.

Now, let us fix a class $\beta\in H_2(X,\mathbb{Z})$. By \cite[Theorem 2]{FP} there exists a smooth, proper, and separated Deligne-Mumford stack $\overline{\mathcal{M}}_{0,n}(X,\beta)$ parametrizing isomorphism classes of stable maps $[C,(x_{1},...,x_{n}),\alpha]$ such that $\alpha_{*}[C] = \beta$. Furthermore, by \cite[Corollary 1]{KP} the coarse moduli space $\overline{M}_{0,n}(X,\beta)$ associated to the stack $\overline{\mathcal{M}}_{0,n}(X,\beta)$ is a normal, irreducible, projective variety with at most finite quotient singularities of dimension
$$
\dim(\overline{M}_{0,n}(X,\beta)) = \dim(X)+\beta\cdot c_1(T_X)+n-3.
$$
The variety $\overline{M}_{0,n}(X,\beta)$ is called the \textit{moduli space of stable maps}, or the \textit{Kontsevich moduli space} of stable maps of class $\beta$ from a rational pre-stable $n$-pointed curve to $X$. 

\subsubsection*{Kontsevich spaces of conics in Grassmannians}
We will denote by $\overline{M}_{0,0}(\mathbb{G}(k,n),2)$ the moduli space of degree two stable maps to the Grassmannian $\mathbb{G}(k,n)$ parametrizing $k$-planes in $\mathbb{P}^n$ embedded via the Pl\"ucker embedding. Now, following \cite[Section 2]{CC10} we are going to describe divisor classes on $\overline{M}_{0,0}(\mathbb{G}(k,n),2)$.  Fix projective subspaces $\Pi^{n-k},\Pi^{n-k-2}\subset\mathbb{P}^n$ of dimension $n-k$ and $n-k-2$, and consider the Schubert cycles
$$
\begin{array}{lll}
\sigma_{1,1}^{k,n} & = & \{W\in\mathbb{G}(k,n)\: | \: \dim(W\cap\Pi^{n-k})\geq 1\};\\ 
\sigma_{2}^{k,n} & = & \{W\in\mathbb{G}(k,n)\: | \: \dim(W\cap\Pi^{n-k-2})\geq 0\}.
\end{array} 
$$
Let $\pi:\overline{M}_{0,1}(\mathbb{G}(k,n),2)\rightarrow\overline{M}_{0,0}(\mathbb{G}(k,n),2)$ be the forgetful morphism and $ev:\overline{M}_{0,1}(\mathbb{G}(k,n),2)\rightarrow\mathbb{G}(k,n)$ the evaluation morphism. We define
$$H_{\sigma_{1,1}}^{k,n} = \pi_{*}ev^{*}\sigma_{1,1}, \: H_{\sigma_2}^{k,n} = \pi_{*}ev^{*}\sigma_2.$$
Furthermore, we will denote by $T^{k,n}$ the class of the divisor of conics that are tangent to a fixed hyperplane section of $\mathbb{G}(k,n)$.

Let $D_{deg}^{k,n}$ be the class of the divisor of maps $[C,\alpha]\in \overline{M}_{0,0}(\mathbb{G}(k,n),2)$ such that the projection of the span of the linear spaces parametrized by $\alpha(C)$ from a fixed subspace of dimension $n-k-2$ has dimension less than $k+2$.

Next we define the divisor class $D_{unb}^{k,n}$. A stable map $\alpha:\mathbb{P}^1\rightarrow\mathbb{G}(k,n)$ induces  a rank $k+1$ subbundle $\mathcal{E}_{\alpha}\subset \mathcal{O}_{\mathbb{P}^1}\otimes K^{n+1}$. If $k = 1$ we define $D_{unb}^{k,n}$ as the closure of the locus of maps $[\mathbb{P}^1,\alpha]\in \overline{M}_{0,0}(\mathbb{G}(k,n),2)$ such that $\mathcal{E}_{\alpha}\neq \mathcal{O}_{\mathbb{P}^1}(-1)^{\oplus 2}$. If $k\geq 2$ there is a trivial subbundle $\mathcal{O}_{\mathbb{P}^1}^{\oplus k-1}\subset\mathcal{E}_{\alpha}$ which induces a $(k-2)$-dimensional subspace $H_{\alpha}\subset\mathbb{P}^n$. In this way we get a map
$$
\begin{array}{lclc}
\xi: & \overline{M}_{0,0}(\mathbb{G}(k,n),2) & \dashrightarrow & \mathbb{G}(k-2,n)\\ 
 & [\mathbb{P}^1,\alpha] & \mapsto & H_{\alpha}
\end{array} 
$$
We define $D_{unb}^{k,n} = \xi^{*}\mathcal{O}_{\mathbb{G}(k-2,n)}(1)$ that is $D_{unb}^{k,n}$ is the closure of the locus of maps $[\mathbb{P}^1,\alpha]\in \overline{M}_{0,0}(\mathbb{G}(k,n),2)$ such that $H_{\alpha}$ intersects a fixed $(n-k+1)$-dimensional subspace of $\mathbb{P}^n$.

Finally, we denote by $\Delta^{k,n}$ the boundary divisor parametrizing stable maps with reducible domain. 

The connection between $\overline{M}_{0,0}(\mathbb{G}(1,3),2)$ and the space of complete quadrics $\mathcal{Q}(3)$ is due to \cite[Lemma 21]{Ce15} which states that there is a finite morphism of degree two
\stepcounter{thm}
\begin{equation}\label{2to1}
\phi:\overline{M}_{0,0}(\mathbb{G}(1,3),2)\rightarrow\mathcal{Q}(3)
\end{equation}
which maps a smooth conic $C\subset \mathbb{G}(1,3)$ to the quadric surface $\bigcup_{[L]\in C}L\subset\mathbb{P}^3$.

\subsubsection*{Kontsevich spaces of conics in Lagrangian Grassmannians}\label{K_Lag}
The Lagrangian Grassmannian $LG(r,2r)\subset\mathbb{G}(r-1,2r-1)$ parametrizes $r$-dimensional subspaces of $K^{2r}$ which are isotropic with respect to the standard symplectic form $\Omega$ in (\ref{symform}). By \cite[Section 2.1]{Te05} $LG(r,2r)$ is an irreducible variety of dimension $\frac{r(r+1)}{2}$ and of Picard rank one. Moreover, the restriction of the Pl\"ucker embedding of $\mathbb{G}(r-1,2r-1)$ yields the minimal homogeneous embedding of $LG(r,2r)$.

In this section we will study the moduli space $\overline{M}_{0,0}(LG(r,2r),2)$ parametrizing conics in $LG(r,2r)$. Let $\mathcal{E}$ be the universal quotient bundle on $\mathbb{G}(r-1,2r-1)$. The Lagrangian Grassmannian $LG(r,2r)\subset\mathbb{G}(r-1,2r-1)$ is the zero locus of a section of $\bigwedge^2\mathcal{E}$ which has first Chern class $(r-1)c_1(\mathcal{O}_{\mathbb{G}(r-1,2r-1)}(1))$. Hence the canonical bundle of $LG(r,2r)$ is given by $\omega_{LG(r,2r)}\cong \mathcal{O}_{LG(r,2r)}(-r-1)$, and $\dim(\overline{M}_{0,0}(LG(r,2r),2)) = \frac{r^2+5r-2}{2}$.

\begin{Remark}\label{comlg}
We recall some facts about the cohomology of $LG(r,2r)$. For details we refer to \cite[Section 3]{BKT03}. Consider a flag $F^1 \subset F^2 \subset \dots \subset F^r \subset K^{2r}$, where $F^j$ are isotropic subspaces of $K^{2r}$ of dimension $j$. Let $\mathcal{D}_r$ be the set of strict partitions $\lambda=(\lambda_1, \dots, \lambda_l)$ with $0 < \lambda_l < \dots < \lambda_1 \le r$ and denote by $\left|\lambda \right|= \lambda_1 + \dots + \lambda_l$ the weight of $\lambda$. For each $\lambda \in \mathcal{D}_r$ there is a codimension $\left|\lambda \right|$ Schubert variety $\Sigma^r_\lambda \subseteq LG(r,2r)$ defined by 
$$\Sigma^r_\lambda:=\{W \in LG(r,2r),\, \dim(W \cap F^{r+1-\lambda_i}) \ge i,\; i=1, \dots,l\}.$$
The class of the Schubert variety $\Sigma^r_\lambda$ in the cohomology ring $H^*(LG(r,2r),\mathbb{Z})$ will be denoted by $\sigma^r_\lambda$. We have that
$$H^*(LG(r,2r),\mathbb{Z})= \bigoplus_{\lambda \in \mathcal{D}_r}\mathbb{Z}\cdot \sigma^r_\lambda$$
with the following relations:
\stepcounter{thm}
\begin{equation}\label{relcomlg}
(\sigma_i^{r})^{2}+2\sum_{k=1}^{r-i}(-1)^{k}\sigma^r_{i+k}\sigma^r_{i-k}=0
\end{equation}
where by convention $\sigma^r_0=1$ and $\sigma^r_i=0$ for $i < 0$.
\end{Remark}

Now, we define divisor classes on $\overline{M}_{0,0}(LG(r,2r),2)$. We denote by $\Delta^r$, the boundary divisor parametrizing stable maps with reducible domain, this is the restriction to $\mml$ of the divisor $\Delta^{r-1,2r-1}$ on $\overline{M}_{0,0}(\mathbb{G}(r-1,2r-1),2)$.

Fix an isotropic subspace $F^{r-1}$ of dimension $r-1$, and consider the divisor $H_{\sigma_2}^r = \pi_{*}ev^{*}\sigma^r_2$, where $\pi: \overline{M}_{0,1}(LG(r,2r),2) \rightarrow \mml$ is the forgetful morphism, $ev: \overline{M}_{0,1}(LG(r,2r),2) \rightarrow LG(r,2r)$ is the evaluation morphism, and $\sigma_2^r$ is the Schubert cycle corresponding to the Schubert variety
$$\Sigma^r_2:=\{W \in LG(r,2r), \dim(W \cap F^{r-1}) \ge 1\}.$$ 
By Remark \ref{comlg}, in $LG(r,2r)$ the only Schubert cycle of codimension two is $\sigma^r_2$, so by \cite[Theorem 1]{Op05} we get that $\Delta^r$ and $H_{\sigma_2}^r$ generate the Picard group of $\mml$. Furthermore, we have that both the divisors $H^{r-1,2r-1}_{\sigma_{1,1}}$ and $H^{r-1,2r-1}_{\sigma_2}$ of $\overline{M}_{0,0}(\mathbb{G}(r-1,2r-1),2)$ restrict to $H_{\sigma_2}^r$ on $\mml$. Then, also $D_{deg}^{r-1,2r-1}$ and $D_{unb}^{r-1,2r-1}$ restrict to the same divisor $D_{unb}^r$ on $\mml$.

Finally, we will denote by $T^r$ the restriction of the divisor $T^{r-1,2r-1}$ to $\mml$, this is the class of the divisor of conics that are tangent to a fixed hyperplane section of $LG(r,2r)$.

\begin{Proposition}\label{propemblg}
Consider the subspaces $H=\{x_2= \dots = x_{r-1}=x_{r+2}=\dots=x_{2r-1} =0\}$ and $\Pi^{r-3}=\{x_0 = \dots = x_{r+1}=0\}$ in $\mathbb{P}^{2r-1}$. There is an embedding
$$
\begin{array}{ccll}
i:& LG(2,H) & \hookrightarrow & LG(r,2r) \\ 
  & L & \mapsto & \langle L , \Pi^{r-3} \rangle 
\end{array} 
$$
which induces an embedding $j:\mmle \rightarrow \mml$. Moreover, the pull-back map $j^{*}:\Pic(\mml)\rightarrow\Pic(\mmle)$ is an isomorphism.
\end{Proposition}
\begin{proof}
Since $\Pi^{r-3}$ is the projectivization of an isotropic subspace of $K^{2r}$, and disjoint from $H$, the map $i$ is well-defined. By \cite[Theorem 1]{Op05} the Picard group of $\mml$ is generated by $\Delta^r$ and $H_{\sigma_{2}}^r$.

Furthermore, we have that $i^*(\sigma_2^r)= \sigma_2^2$ and then $j^*(H^r_{\sigma_2^r})=H^2_{\sigma_2^2}$. Finally, since $j^*(\Delta^r)=\Delta^2$ we conclude that the pull-back map is an isomorphism.
\end{proof}

\begin{Lemma}\label{sym_symp}
Let $C_1, C_2\subset\mathbb{G}(1,3)$ be two smooth conics corresponding to the rulings $\bigcup_{[L]\in C_1}L$ and $\bigcup_{[L]\in C_2}L$ of a smooth quadric $Q\subset\mathbb{P}^3$. The following are equivalent:
\begin{itemize}
\item[(a)] $C_1$ is contained in $LG(2,4)$ but $C_2$ is not;
\item[(b)] the lines in the ruling $\bigcup_{[L]\in C_1}L$ are Lagrangian while the general line in the ruling $\bigcup_{[L]\in C_2}L$ is not;
\item[(c)] the matrix of $Q$ has a scalar multiple that is symplectic.
\end{itemize}
\end{Lemma}
\begin{proof}
The actions of $Sp(4)$ on $\overline{M}_{0,0}(LG(2,4),2)$ in (\ref{actKLG}) and on $\mathcal{S}_4$ in (\ref{acsimp}) are compatible. Therefore, it is enough to prove that the equivalence of the conditions in statement holds for a particular smooth quadric. 

Consider the quadric $Q = \{x_0^2+x_1^2-x_2^2-x_3^2 = 0\}\subset\mathbb{P}^3$. If $M_{Q}$ is the matrix of $Q$ we have $M_{Q}^t\Omega M_{Q} = -\Omega$, and hence $iM_{Q}$ is symplectic. 

Now, one of the rulings of $Q$ is given by the following lines
$$L_{s,t} = \left\langle (t,-s,-t,s), (s,t,s,t)\right\rangle$$  
with $[s:t]\in\mathbb{P}^1$. Note that $L_{s,t}$ is Lagrangian with respect to $\Omega$ for all $[s:t]\in\mathbb{P}^1$. 

Fix homogeneous coordinates $[Z_0:\dots :Z_5]$ on $\mathbb{P}^5$. The Lagrangian Grassmannian $LG(2,4)$ is cut out on the Grassmannian $\mathbb{G}(1,3)$ by the hyperplane $H = \{Z_1+Z_4 = 0\}$. Via the Pl\"ucker embedding the ruling $L_{s,t}$ corresponds to the conic given by the image of the following morphism
$$
\begin{array}{ccl}
\mathbb{P}^1 & \longrightarrow & \mathbb{G}(1,3)\\ 
(s,t) & \longmapsto & [t^2+s^2:2st:t^2-s^2:-s^2+t^2:-2st:-t^2-s^2] 
\end{array} 
$$
which therefore is contained in $H\cap\mathbb{G}(1,3) = LG(2,4)$. The other ruling of $Q$ is given by 
$$R_{u,v} = \left\langle (u,-v,u,v), (v,u,-v,u)\right\rangle$$  
with $[u:v]\in\mathbb{P}^1$. The corresponding conic is given by the image of  
$$
\begin{array}{ccl}
\mathbb{P}^1 & \longrightarrow & \mathbb{G}(1,3)\\ 
(u,v) & \longmapsto & [u^2+v^2:-2uv:u^2-v^2:v^2-u^2:-2uv,u^2+v^2] 
\end{array} 
$$
which is not contained in $H\cap\mathbb{G}(1,3) = LG(2,4)$. Hence, the general line in the ruling $R_{u,v}$ is not Lagrangian.
\end{proof}

\begin{Lemma}\label{sphKLG}
The following $Sp(4)$-action on $\overline{M}_{0,0}(LG(2,4),2)$ 
\stepcounter{thm}
\begin{equation}\label{actKLG}
\begin{array}{cll}
Sp(4) \times \overline{M}_{0,0}(LG(2,4),2) & \longrightarrow & \overline{M}_{0,0}(LG(2,4),2) \\ 
(M, [C, \alpha]) & \longmapsto & [C, \wedge^{r}M\circ\alpha] 
\end{array} 
\end{equation}
gives to $\overline{M}_{0,0}(LG(2,4),2)$ a structure of spherical variety. 
\end{Lemma}
\begin{proof}
By Lemma \ref{sym_symp} a ruling of the quadric $Q = \{x_0^2 + x_1^2 - x_2^2 - x_3^2 =0\}$ yields a conic in $LG(2,4)$. Let $\mathscr{B}\subset Sp(4)$ be the Borel subgroup of the symplectic group in Remark \ref{borelsym}. Note that $\dim(\mathscr{B})=6$. The stabilizer of $Q$ in $\mathscr{B}$ is given by
$$\begin{pmatrix}
A_{2,2} & 0_{2,2} \\
B_{2,2} & A_{2,2}^{-t}
\end{pmatrix}  \begin{pmatrix}
I_{2,2} & 0_{2,2} \\
0_{2,2} & -I_{2,2}
\end{pmatrix} \begin{pmatrix}
A_{2,2}^t & B_{2,2}^t \\
0_{2,2} & A_{2,2}^{-1}
\end{pmatrix} = \begin{pmatrix}
A_{2,2}^tA_{2,2} & A_{2,2}B^t_{2,2} \\
B_{2,2}A^t_{2,2} & B_{2,2}B^t_{2,2} - A_{2,2}^{-t}A_{2,2}^{-1}
\end{pmatrix}$$
So, we get $B_{2,2}=0_{2,2}$ and $A_{2,2}^tA_{2,2} = I_{2,2}$. Then 
$$Stab_{\mathscr{B}}(Q) = \left\lbrace M = \begin{pmatrix}
a & 0 & 0 & 0 \\
0 & b &0 & 0 \\
0 & 0 & \frac{1}{a} & 0 \\
0 & 0 & 0 & \frac{1}{b} \\
\end{pmatrix}; \text{ with } a^2 = b^2 = 1 \right\rbrace$$
and $\dim(Stab_{\mathscr{B}}(Q))=0$.
\end{proof}

\begin{Proposition}\label{isor2}
The restriction of the map in (\ref{2to1}) to $\mmle$ yields an isomorphism
\stepcounter{thm}
\begin{equation}\label{maplg}
\varphi : \mmle \rightarrow \mathcal{S}_4
\end{equation}
where $\mathcal{S}_4$ is the wonderful compactification of the space of symplectic quadrics of $\mathbb{P}^3$.
\end{Proposition}
\begin{proof}
By Lemma \ref{sym_symp} the restriction of the map in (\ref{2to1}) to $\mmle$ yields a $1$-to-$1$ morphism which is surjective since both $\mmle$ and $\mathcal{S}_4$ are $6$-dimensional. 

Finally, since $\mathcal{S}_4$ is smooth and $\mmle$ is normal Zariski's main theorem \cite[Chapter 3, Section 9]{Mum99} yields that the morphism in (\ref{maplg}) is an isomorphism. 
\end{proof}

\begin{Lemma}\label{boucollg}
The divisor classes $\Delta^{2}, D_{unb}^{2}$ and the divisor classes $H_{\sigma_2}^{2},T^{2}$ are respectively the classes of the boundary divisors and the colors of the spherical variety $\overline{M}_{0,0}(LG(2,4),2)$.
\end{Lemma}
\begin{proof}
The actions (\ref{actKLG}) and (\ref{acsimp}) are equivariant with respect to the map $\varphi$ in (\ref{maplg}). So boundary divisors and colors of $\mmle$ are mapped by $\varphi$ to boundary divisors and colors of $\mathcal{S}_4$ respectively. By Proposition \ref{bound_col}, in $\mathcal{S}_4$ the colors are $D_1,D_2$ and the boundary divisors are $E_1,S_2^{(1)}(\mathcal{V}_2^{3})$. Moreover, $\Delta^{2}, D_{unb}^{2}$ are stabilized by the $Sp(4)$-action in (\ref{actKLG}) and choosing the flag of isotropic linear subspaces $ \{x_0 = x_1 = 0\}\subset\{x_0 = 0\}$ we see that $H_{\sigma_2}^{2},T^{2}$ are stabilized by the action of the Borel subgroup of $Sp(4)$ in Remark \ref{borelsym}. Moreover, it is straightforward to see that the inverse image via the morphism $\varphi$ in (\ref{maplg}) of $S_2^{(1)}(\mathcal{V}_2^{3}), E_1, D_1,D_2$ are divisors of classes $\Delta^{2}, D_{unb}^{2},H_{\sigma_2}^{2},T^{2}$. Now, assume to have another boundary divisor in $\mmle$. Then, $\varphi$ maps this divisor to a boundary divisor of $\mathcal{S}_4$, but the only boundary divisors of $\mathcal{S}_4$ are $S_2^{(1)}(\mathcal{V}_2^{3}), E_1$. Then, the only boundary divisors of $\mmle$ are $\Delta^{2}, D_{unb}^{2}$, and similarly the only colors of $\mmle$ are $H_{\sigma_2}^{2},T^{2}$.
\end{proof}

We denote by $\overline{M}_{0,0}(LG(r,2r),2,1)$ the moduli space of weighted stable maps to $LG(r,2r)$. In this space degree one tails of a stable map are replaced by their attaching point. We refer to \cite{MM07} for the construction of moduli of weighted stable maps. 

\begin{Proposition}\label{effneflg}
The divisors $\Delta^{r}, D_{unb}^{r}$ generate the effective cone of $\mml$, and the divisors $H_{\sigma_2}^r, T^r$ generate the nef cone of $\mml$. 

The divisor $H_{\sigma_2}^r$ induces a birational morphism 
$$f_{H_{\sigma_2}^r}:\mml\rightarrow \widetilde{Chow}(LG(r,2r),2)$$ 
which is an isomorphism away form the locus $Q^r(1)$ of double covers of a line in $LG(r,2r)$, and contracts $Q^r(1)$ so that the locus of double covers with the same image maps to
a point, where $\widetilde{Chow}(LG(r,2r),2)$ is the normalization of the Chow variety of conics in $LG(r,2r)$.

The divisor $T^r$ induces a morphism 
$$f_{T^r}:\mml\rightarrow \overline{M}_{0,0}(LG(r,2r),2,1)$$ 
which is an isomorphism away from $\Delta^r$ and contracts the locus of maps with reducible domain $[C_1\cup C_2,\alpha]$ to $\alpha(C_1\cap C_2)$. Hence, $f_{T^r}$ contracts the divisor $\Delta^r$ onto $LG(r,2r)\subset \overline{M}_{0,0}(LG(r,2r),2,1)$. 
\end{Proposition}
\begin{proof}
By \cite[Proposition 4.5.4.4]{ADHL15} and Lemma \ref{boucollg} the effective cone of $\overline{M}_{0,0}(LG(2,4),2)$ is generated by $\Delta^{2},D_{unb}^{2}, H_{\sigma_2}^{2}, T^{2}$. Consider the isomorphism $\varphi$ in (\ref{maplg}). We have
$$
\varphi^{*}E_1 =D_{unb}^{2},\: \varphi^{*}S_2^{(1)}(\mathcal{V}_2^3) = \Delta^{2},\: \varphi^{*}D_1 = H_{\sigma_2}^{2},\: \varphi^{*}D_2= T^{2}.
$$
Now, the relations among the boundary divisors and the colors of $\mathcal{S}_4$ in Proposition \ref{mcd_4} yield the following relations in the Picard group of $\mmle$:
\stepcounter{thm}
\begin{equation}\label{relPiclg}
H_{\sigma_2}^{2}\sim\frac{\Delta^{2}+2D_{unb}^2}{2},\:  T^{2} \sim \Delta^{2}+D_{unb}^2
\end{equation}
and the statement in the case $r = 2$ follows from Propositions \ref{eff_nef} and \ref{isor2}.

Now, consider the case $r > 2$. Since $T^r$ is the pull-back of $T^{r-1,2r-1}$ via the embedding $\overline{M}_{0,0}(LG(r,2r),2)\hookrightarrow\overline{M}_{0,0}(\mathbb{G}(r-1,2r-1),2)$ \cite[Theorem 3.8]{CC10} yields that $T^r$ induces a morphism 
$$f_{T^r}:\mml\rightarrow \overline{M}_{0,0}(LG(r,2r),2,1)$$ 
which is an isomorphism away from $\Delta^r$ and contracts the locus of maps with reducible domain $[C_1\cup C_2,\alpha]$ to $\alpha(C_1\cap C_2)$. Hence, $f_{T^r}$ contracts the divisor $\Delta^r$ onto $LG(r,2r)\subset \overline{M}_{0,0}(LG(r,2r),2,1)$. So $\Delta^r$ generates an extremal ray of the effective cone, and $T^r$ generates an extremal ray of the nef cone.

Similarly, \cite[Proposition 3.7]{CC10} yields the morphism $f_{H_{\sigma_2}^r}:\mml\rightarrow \widetilde{Chow}(LG(r,2r),2)$, and hence $H_{\sigma_2}^r$ generates the other extremal ray of the nef cone. 

Now, following the proof of \cite[Lemma 3.4]{CC10} we define the class of a curve $\Gamma$ in $\overline{M}_{0,0}(LG(r,2r),2)$ whose deformations cover the whole of $\overline{M}_{0,0}(LG(r,2r),2)$. Consider a general hyperplane section $Z$ of $LG(2,4)\subset\mathbb{P}^4$, and a general line in this hyperplane section. The planes containing the line cut out a pencil of conics on $Z\subset LG(2,4)$. Hence we get a rational curve $C\subset \overline{M}_{0,0}(LG(2,4),2)$ parametrizing these conics. Let $\Gamma$ be the image of $C$ via the embedding in Proposition \ref{propemblg}. Then $H_{\sigma_2}^r\cdot\Gamma = 1$, and $\Delta^r\cdot\Gamma = 2$ since there are two reducible conics in a general pencil of conics in the quadric surface $Z$. Now, by (\ref{relPiclg}) we get that $D_{unb}^r\cdot\Gamma = 0$, and by \cite[Theorem 2.2]{BDPP13} we conclude that $D_{unb}^r$ generates the other extremal ray of the effective cone. 
\end{proof}

\begin{Remark}\label{contr4}
Note that $Q^r(1)$ is a divisor in $\overline{M}_{0,0}(LG(r,2r),2)$ if and only if $r = 2$. By Proposition \ref{isor2} we have $\mmle\cong\mathcal{S}_4$ which by Proposition \ref{x4g14} is the blow-up of $\mathbb{G}(1,4)$ along the Veronese $\mathcal{V}_2^{3}$. In this case 
$$f_{H_{\sigma_2}^2}:\mmle\rightarrow \widetilde{Chow}(LG(2,4),2)$$ 
is nothing but the blow-down morphism $\mathcal{S}_4\rightarrow\mathbb{G}(1,4)$. Indeed, since $LG(2,4)\subset\mathbb{P}^4$ is a quadric hypersurface and hence does not contain any plane we have that all planes in $\mathbb{P}^4$ cut out a conic on $LG(2,4)$. Hence, we may identify the Chow variety of conics in $LG(2,4)$ with $\mathbb{G}(2,4)\cong\mathbb{G}(1,4)$.

Furthermore, by Proposition \ref{mcd_4} the morphism 
$$f_{T^2}:\mmle\rightarrow \overline{M}_{0,0}(LG(2,4),2,1)$$ 
is induced by the strict transform of the restriction to $\mathbb{G}(1,4)$ of the linear system of quadrics in $\mathbb{P}^9$ containing $\mathcal{V}_2^{3}$. In this way we realize $\overline{M}_{0,0}(LG(2,4),2,1)$ as a $6$-fold of degree $40$ in $\mathbb{P}^{14}$ which is singular along a $3$-fold isomorphic to $LG(2,4)$. 
\end{Remark}

\begin{thm}\label{mcd_lg}
The Mori chamber decomposition of $\Eff(\mml)$ has three chambers as displayed in the following picture:
$$
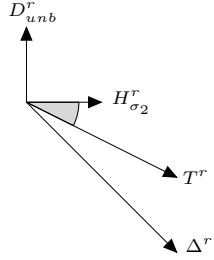
\begin{tikzpicture}[line cap=round,line join=round,>=triangle 45,x=1.0cm,y=1.0cm]
xmin=2.5,
xmax=7.5,
ymin=1.7032313370047312,
ymax=5.301934941656083,
xtick={2.5,3.0,...,5.5},
ytick={2.0,2.5,...,5.0},]
\clip(2.1,2.0) rectangle (7.5,5.4);
\draw [->,line width=0.1pt] (3.,4.) -- (3.,5.); 
\draw [->,line width=0.1pt] (3.,4.) -- (4.,4.); 
\draw [->,line width=0.1pt] (3.,4.) -- (5.,3.); 
\draw [->,line width=0.1pt] (3.,4.) -- (5.,2.); 
\draw [shift={(3.,4.)},line width=0.4pt,fill=black,fill opacity=0.15000000596046448]  (0,0) --  plot[domain=-0.46364760900080615:0.,variable=\t]({1.*0.6965067581669063*cos(\t r)+0.*0.6965067581669063*sin(\t r)},{0.*0.6965067581669063*cos(\t r)+1.*0.6965067581669063*sin(\t r)}) -- cycle ;
\begin{scriptsize}
\draw[color=black] (3.085808915158281,5.2) node {$D_{unb}^r$};
\draw[color=black] (4.4,4) node {$H_{\sigma_2}^r$};
\draw[color=black] (5.25,3) node {$T^r$};
\draw[color=black] (5.3,2.1) node {$\Delta^r$};
\end{scriptsize}
\end{tikzpicture}
$$
where $H_{\sigma_2}^{r}\sim \frac{1}{2}(\Delta^r+2D_{unb}^r)$ and $T^r\sim\Delta^r+D_{unb}^r$. Furthermore, $\Mov(\overline{M}_{0,0}(LG(r,2r),2))$ is generated by $T^r$ and $D_{unb}^r$ if $r > 2$, while $\Mov(\overline{M}_{0,0}(LG(2,4),2))$ is generated by $T^2$ and $H_{\sigma_2}^2$. The Cox ring $\Cox(\overline{M}_{0,0}(LG(2,4),2))$ is generated by the sections of $\Delta^2,D_{unb}^2,H_{\sigma_2}^2,T^2$. 

The birational model $X_r$ corresponding to the chamber delimited by $H_{\sigma_2}^r$ and $D_{unb}^r$ is a fibration $X_r\rightarrow SG(r-2,2r)$ with fibers isomorphic to $\mathbb{G}(2,4)$, where $SG(r-2,2r)$ is the symplectic Grassmannian parametrizing isotropic subspaces of dimension $r-2$. Finally, $D_{unb}^r$ contracts $\overline{M}_{0,0}(LG(r,2r),2)$ onto $SG(r-2,2r)$.
\end{thm}
\begin{proof}
First consider the case $r = 2$. The statement on the generators of the Cox ring follows from Proposition \ref{effneflg} and Remark \ref{gen_cox}. Furthermore, by Remarks \ref{toric} and \ref{gen_cox} the Mori chamber decomposition of $\Eff(\overline{M}_{0,0}(LG(2,4),2))$ is a, possibly trivial, coarsening of the decomposition in the statement. Since by Proposition \ref{effneflg} the effective cone $\Eff(\overline{M}_{0,0}(LG(2,4),2))$ is generated by $\Delta^{2}$ and $D_{unb}^{2}$, and $H_{\sigma_2}^2, T^2$ generate $\Nef(\overline{M}_{0,0}(LG(2,4),2))$ no ray can be removed, and the Mori chamber decomposition is as in the statement. The relations $H_{\sigma_2}^{r}\sim \frac{1}{2}(\Delta^r+2D_{unb}^r)$ and $T^r\sim\Delta^r+D_{unb}^r$ follow from the proof of Proposition \ref{propemblg} and (\ref{relPiclg}).

Now, consider the case $r >2$. By Proposition \ref{effneflg} the wall-crossing of $T^r$ induces a divisorial contraction, and a divisor inside the chamber delimited by $T^r$ and $H^r_{\sigma_2}$ is ample. By Proposition \ref{effneflg} the wall-crossing of $H_{\sigma_2}^r$ yields a birational contraction whose exceptional locus is the variety $Q^r(1)$ of double covers of a line in $LG(r,2r)$.

Next, we will construct the birational model of $\mml$ corresponding to the chamber delimited by $H_{\sigma_2}^r$ and $D_{unb}^r$. Let $H\subset\mathbb{P}^{2r-1}$ be an $(r+1)$-plane containing an isotropic $(r-1)$-plane $\Pi\subset\mathbb{P}^{2r-1}$. Then $\Pi = \Pi^{\perp}\supset H^{\perp}$. So $H^{\perp}\subset H$. Now, the $(r+1)$-planes containing their orthogonal are in bijection with the $(r-3)$-planes of $\mathbb{P}^{2r-1}$ that are isotropic. The variety parametrizing such $(r-3)$-planes is the symplectic Grassmannian $SG(r-2,2r)$. Let $\mathcal{U}_r$ be the universal bundle on $SG(r-2,2r)$, $\mathcal{U}_r^{\perp}\subset \mathcal{U}_r$ its orthogonal, and $\mathcal{Q}_r = \mathcal{U}_r/\mathcal{U}_r^{\perp}$ the quotient bundle. Then $\mathcal{Q}_r$ has rank four, and we may consider the relative Lagrangian Grassmannian $LG(2,\mathcal{Q}_r)\rightarrow SG(r-2,2r)$, and the relative Hilbert scheme $\Hilb_2(LG(2,\mathcal{Q}_r))\rightarrow SG(r-2,2r)$. Note that since $LG(2,4)$ does not contain planes the fibers of $\Hilb_2(LG(2,\mathcal{Q}_r))\rightarrow SG(r-2,2r)$ are isomorphic to $\mathbb{G}(2,4)$. Indeed, we can associate to a plane in $\mathbb{P}^4$ the conic it cuts out on $LG(2,4)$. Set $X_r := \Hilb_2(LG(2,\mathcal{Q}_r))\rightarrow SG(r-2,2r)$. Note that
$$\dim(X_r) = \dim(SG(r-2,2r))+6 = 2r^2-4r-\frac{3(r-2)^2-r+2}{2} + 6 = \frac{r^2+5r-2}{2} = \dim(\mml)$$
and there is a birational transformation $\mml\dasharrow X_r$ inducing an isomorphism between the complement of $Q^r(1)$ in $\mml$ and the complement of the locus of double lines in $X_r$. Since $r> 2$ both these loci are in codimension greater that one. Furthermore, $H_{\sigma_2}^r$ induces a morphism on $X_r$ associating to a conic the reduced curve on which it is supported. Hence, this morphism is birational and contracts the locus of double lines. Finally $D_{unb}^r$ induces on $X_r$ the fibration $X_r \rightarrow SG(r-2,2r)$. Indeed, this fibration yields the rational fibration $\mml\dasharrow SG(r-2,2r)$ associating to a stable map that is not $2$-to-$1$ onto a line the orthogonal of the $(r+1)$-plane in $\mathbb{P}^{2r-1}$ generated by the $(r-1)$-planes parametrized by the image of the map. Hence, the cone generated by $H_{\sigma_2}^r$ and $D_{unb}^r$ is the nef cone of $X_r$. 

Finally, the claim about the movable cones follows from Remark \ref{contr4} since $H_{\sigma_2}^2$ induces a divisorial contraction, while for $r > 2$ the divisor $H_{\sigma_2}^2$ yields a small contraction and $D_{unb}^r$ induces a non trivial fibration. 
\end{proof}

We now study the positivity of the anti-canonical divisor of $\overline{M}_{0,0}(LG(r,2r),2)$.

\begin{Proposition}\label{can}
Let $\overline{\mathcal{M}}_{0,0}(LG(r,2r),2)$ be the smooth Deligne-Mumford stack of degree two stable maps to $LG(r,2r)$, $\overline{H}_{\sigma_2}^r,\overline{T}^r,\overline{\Delta}^r,\overline{D}_{unb}^r$ the divisors on $\overline{\mathcal{M}}_{0,0}(LG(r,2r),2)$ corresponding to $H_{\sigma_2}^r,T^r,\Delta^r,D_{unb}^r$ respectively.

The anti-canonical divisor of the stack $\overline{\mathcal{M}}_{0,0}(LG(r,2r),2)$ is given by
$$-K_{\overline{\mathcal{M}}_{0,0}(LG(r,2r),2)} = 5\overline{H}_{\sigma_2}^r +\frac{r-7}{2}\overline{D}_{unb}^r$$ 
for $r> 2$, while $-K_{\overline{\mathcal{M}}_{0,0}(LG(2,4),2)} = 5\overline{H}_{\sigma_2}^2 -5\overline{D}_{unb}^2$. Furthermore, the anti-canonical divisor of $\mmle$ is given by 
$$-K_{\overline{M}_{0,0}(LG(r,2r),2)} = 5H_{\sigma_2}^r +\frac{r-7}{2}D_{unb}^r$$
for $r>2$, while for $r = 2$ we have that
$$-K_{\mmle} = 5H_{\sigma_2}^2 -2 D_{unb}^2.$$
\end{Proposition}
\begin{proof}
We will compute the canonical divisor of $\overline{\mathcal{M}}_{0,0}(LG(r,2r),2)$ using the formula in \cite[Theorem 1.1]{dJS17}. Hence, we need the Chern classes $c_1(T_{LG(r,2r)}), c_2(T_{LG(r,2r)})$, where $T_{LG(r,2r)}$ is the tangent bundle of $LG(r,2r)$. Recall that $T_{LG(r,2r)}\cong \Sym^2(S^{\vee})$, where $S$ is the universal bundle.

Let us pretend that $S^{\vee} = L_1\oplus\dots \oplus L_r$ splits as direct sum of line bundles. We will then use Whitney's formula along with the splitting principle to compute the Chern classes of $\Sym^2(S^{\vee})$. Set $c_1(L_i) = \alpha_i$ for $i = 1,\dots, r$. Then 
$$c(S^{\vee}) = \prod_{i = 1}^{r}(1+\alpha_i)$$
and hence
\begin{equation}\label{chern1}
c_1(S^{\vee}) = \alpha_1+\dots+\alpha_r,\quad c_2(S^{\vee}) = \alpha_1\alpha_2+\dots+\alpha_1\alpha_r+\alpha_2\alpha_3+\dots+\alpha_{r-1}\alpha_r.
\end{equation}
Furthermore 
$$\Sym^2(S^{\vee}) = L_1^{\otimes 2}\oplus (L_1\otimes L_2)\oplus\dots\oplus (L_1\otimes L_r)\oplus L_2^{\otimes 2}\oplus\dots \oplus L_r^{\otimes 2}
$$
yields 
$$
\begin{array}{ll}
c(\Sym^2(S^{\vee})) = & (1+2\alpha_1)(1+\alpha_1+\alpha_2)\dots (1+\alpha_1+\alpha_r)(1+2\alpha_2)\dots (1+2\alpha_r) =\\ 
 & 1+(r+1)\sum_{i=1}^r\alpha_i+\frac{r^2+r-2}{2}\sum_{i=1}^{r}\alpha_i^2+(r^2+2r)(\alpha_1\alpha_2+\dots+\alpha_{r-1}\alpha_r)+\dots=\\
 & 1 +(r+1)\sum_{i=1}^r\alpha_i + \frac{r^2+r-2}{2}(\sum_{i=1}^r\alpha_i)^2+(r+2)(\alpha_1\alpha_2+\dots+\alpha_{r-1}\alpha_r)+\dots=\\
 & 1 +(r+1)c_1(S^{\vee}) + \frac{r^2+r-2}{2}c_1(S^{\vee})^2+(r+2)c_2(S^{\vee})+\dots
\end{array} 
$$
where in the last equality we plugged-in the formulas in (\ref{chern1}). Recall that $c_1(S^{\vee}) = \sigma_1^r$, $c_2(S^{\vee}) = \sigma_{2}^r$ and that by (\ref{relcomlg}) we have $(\sigma_1^r)^2 = 2\sigma_{2}^r$. Hence
$$c_1(T_{LG(r,2r)}) = (r+1)\sigma_1^r,\quad c_2(T_{LG(r,2r)}) = (r^2+2r)\sigma_{2}^r.$$
Now, plugging-in these formulas in \cite[Theorem 1.1]{dJS17} we get
$$
K_{\overline{\mathcal{M}}_{0,0}(LG(r,2r),2)} = -\frac{2r+6}{4}\overline{H}_{\sigma_2}^r+\frac{r-7}{4}\overline{\Delta}^r.
$$
Let $\pi:\overline{\mathcal{M}}_{0,0}(LG(r,2r),2)\rightarrow\overline{M}_{0,0}(LG(r,2r),2)$ be the canonical morphism from $\overline{\mathcal{M}}_{0,0}(LG(r,2r),2)$ to its coarse moduli space. Note that $\pi:\overline{\mathcal{M}}_{0,0}(LG(r,2r),2)\rightarrow\overline{M}_{0,0}(LG(r,2r),2)$ is an isomorphism in codimension one for all $r > 2$, while for $r=2$ it is ramified on the divisor $D_{unb}^2$. When $r = 2$ the stack has non trivial inertia along the divisor $\overline{D}_{unb}^2$ since a general stable map in $\overline{D}_{unb}^2$ has automorphism group $\mathbb{Z}/2\mathbb{Z}$. Taking this into account we get that $\pi^{*}D_{unb}^2 = 2\overline{D}_{unb}^2$, and hence Theorem \ref{mcd_lg} yields $\overline{\Delta}^r = 2\overline{H}_{\sigma_2}^r-2\overline{D}_{unb}^r$ if $r > 2$, and $\overline{\Delta}^2 = 2\overline{H}_{\sigma_2}^2-4\overline{D}_{unb}^2$. So, in terms of $\overline{H}_{\sigma_2}^r$ and $\overline{D}_{unb}^r$ the canonical divisor of the stack is given by 
$$K_{\overline{\mathcal{M}}_{0,0}(LG(r,2r),2)} = -5\overline{H}_{\sigma_2}^r-\frac{r-7}{2}\overline{D}_{unb}^r$$
if $r> 2$ , and $K_{\overline{\mathcal{M}}_{0,0}(LG(2,4),2)} = -5\overline{H}_{\sigma_2}^2+5\overline{D}_{unb}^2$.
Furthermore, when $r>2$ the formula above gives the expression of the canonical divisor of $\overline{M}_{0,0}(LG(r,2r),2)$ in the statement since $\overline{M}_{0,0}(LG(r,2r),2)$ and $\overline{\mathcal{M}}_{0,0}(LG(r,2r),2)$ are isomorphic in codimension one for $r > 2$. 

However, when $r = 2$ we have that 
$$K_{\overline{\mathcal{M}}_{0,0}(LG(2,4),2)} = \pi^{*}K_{\overline{M}_{0,0}(LG(2,4),2)}+\overline{D}_{unb}^2.$$ 
Let us write $K_{\overline{M}_{0,0}(LG(2,4),2)} = -5H_{\sigma_2}^2+a D_{unb}^2$. Recalling that $\pi^{*}D_{unb}^2 = 2\overline{D}_{unb}^2$ we get
$$-5\overline{H}_{\sigma_2}^2+5\overline{D}_{unb}^2 = K_{\overline{\mathcal{M}}_{0,0}(LG(2,4),2)} =  \pi^{*}(-5H_{\sigma_2}^2+a D_{unb}^2)+\overline{D}_{unb}^2 = -5\overline{H}_{\sigma_2}^2+(2a+1)\overline{D}_{unb}^2.$$
Hence, $a = 2$ and $K_{\overline{M}_{0,0}(LG(2,4),2)} = -5H_{\sigma_2}^2+2D_{unb}^2$.
\end{proof}

\begin{Remark}
Since $\omega_{\mathbb{G}(1,4)} = \mathcal{O}_{\mathbb{G}(1,4)}(-5)$ and $\codim_{\mathbb{G}(1,4)}(\mathcal{V}_2^3)=3$ the formula $K_{\overline{M}_{0,0}(LG(2,4),2)} = -5H_{\sigma_2}^2+2D_{unb}^2$ can also be deduced from the description of $\overline{M}_{0,0}(LG(2,4),2)$ as the blow-up of $\mathbb{G}(1,4)$ along $\mathcal{V}_2^3$ in Proposition \ref{isor2}.
\end{Remark}

\begin{Corollary}\label{Fano}
The moduli space $\mml$ is Fano for $2\leq r\leq 6$, weak Fano for $r = 7$, and $-K_{\mml}$ is not ample for $r\geq 8$.
\end{Corollary}
\begin{proof}
By Propositions \ref{mcd_lg} and \ref{can} we have that $-K_{\mml}$ is a multiple of $H_{\sigma_2}^r$ if $r = 7$. Furthermore, $-K_{\mml}$ lies in the interior of $\Nef(\mml)$ for $2\leq r\leq 6$, while for $r\geq 8$ we have that $-K_{\mml}$ lies in the interior of the cone generated by $H_{\sigma_2}^r$ and $D_{unb}^r$.
\end{proof}

Finally, the following result on automorphisms of $\mmle$ is at hand. 
\begin{Corollary}\label{aut}
The automorphism group of $\mmle$ is given by
$$\PsAut(\mmle)\cong \Aut(\mmle) \cong PSp(4)$$
where $PSp(4)$ is the projective symplectic group, and $\PsAut(\mmle)$ is the group of birational self-maps of $\mmle$ inducing automorphisms in codimension one. 
\end{Corollary}
\begin{proof}
By Propositions \ref{x4g14}, \ref{isor2} we have that $\mmle$ is isomorphic to the blow-up of $\mathbb{G}(1,4)$ along the Veronese $\mathcal{V}_2^3$. Let $\varphi\in \Aut(\mmle)$ be an automorphism. Then either $\phi$ preserves the two extremal rays of $\Eff(\mmle)$ in Theorem \ref{mcd_lg} or it swaps them. In the second case $\phi$ must swap also the extremal rays of $\Nef(\mmle)$ but this is not possible since for instance $T^2$ has more sections than $H^2_{\sigma_2}$. Therefore, $\phi$ stabilizes the exceptional divisor $D_{unb}^2$ of the blow-up and then it induces an automorphism $\overline{\phi}$ of $\mathbb{G}(1,4)$ that stabilizes $\mathcal{V}_2^3$.

Now, the automorphism group of $\mathbb{G}(1,4)$ is isomorphic to $PGL(5)$ and all these automorphisms are induced by automorphisms of the ambient projective space $\mathbb{P}^9$ \cite[Theorem 1.1]{Co89}. The restriction of $\overline{\phi}$ to $\mathcal{V}_2^3$ yields an automorphism $\overline{\phi}_{|\mathcal{V}_2^3}$ of $\mathbb{P}^3$. Since $\overline{\phi}$ is an automorphism of $\mathbb{G}(1,4)$, which we interpret as the closure of the space of symplectic and symmetric matrices modulo scalar, the restriction $\overline{\phi}_{|\mathcal{V}_2^3}\in PGL(4)$ must map symplectic matrices to symplectic matrices. Hence, $\overline{\phi}_{|\mathcal{V}_2^3}\in PSp(4)$. So, we get a morphism of groups
$$
\begin{array}{ccll}
\chi :& \Aut(\mmle) & \rightarrow & PSp(4)\\ 
  & \phi & \mapsto & \overline{\phi}_{|\mathcal{V}_2^3}
\end{array}
$$ 
which is surjective. Now, if $\overline{\phi}_{|\mathcal{V}_2^3}$ is the identity it must be the restriction of the identity automorphism of the ambient projective space $\mathbb{P}^9$ in which both $\mathcal{V}_2^3$ and $\mathbb{G}(1,4)$ are embedded. Since  $\mathbb{G}(1,4)$ and $\mmle$ are birational we get that $\overline{\phi}_{|\mathcal{V}_2^3}$ must come from the identity of $\Aut(\mmle)$, and hence $\chi$ is an isomorphism. Finally, since by Proposition \ref{isor2} and Corollary \ref{Fano} $\mmle$ is a smooth Fano variety the result on $\PsAut(\mmle)$ follows from \cite[Proposition 7.2]{Ma18a}.
\end{proof}


\bibliographystyle{amsalpha}
\bibliography{Biblio}

\end{document}